\pdfoutput=1
\documentclass[11pt]{amsart}

\usepackage{amsthm, amsfonts, graphicx, pinlabel}
\usepackage{amssymb}
\usepackage[all]{xy}
\usepackage{hyperref}
\usepackage{mathtools}
\usepackage[usenames,dvipsnames]{xcolor}
\usepackage{enumitem}
\usepackage{tikz-cd}
\newcommand{\R}{\mathbb{R}}
\newcommand{\Z}{\mathbb{Z}}
\newcommand{\modsp}{\mathcal{M}(p_1,p_2;p_0 | s)}
\newcommand{\bmodsp}{\overline{\mathcal{M}}(p_1,p_2;p_0 | s)}
 \newcommand{\M}{\mathcal{M}}

\newcommand{\w}[3]{\ensuremath{w_{{#1},{#2};{#3}}}}
\newcommand{\V}[1]{\ensuremath{C^{#1}(F)}}

\newcommand{\ev}[4]{\operatorname{ev}_{{#1},{#2};{#3}}^{#4}}
\newcommand{\Mto}[4]{\mathcal{M}_{{#1},{#2};3}(P_{#3},{#4})}
\newcommand{\Mfro}[4]{\mathcal{M}_{{#1},{#2};3}({#4},P_{#3})}
\newcommand{\Mtofro}[4]{\mathcal{M}_{{#1},{#2};3}({#3},{#4})}
\newcommand{\Mtob}[4]{\overline{\mathcal{M}}_{{#1},{#2};3}(P_{#3},{#4})}
\newcommand{\Mfrob}[4]{\overline{\mathcal{M}}_{{#1},{#2};3}({#4},P_{#3})}

\DeclareMathOperator{\ind}{ind}

\DeclareMathOperator{\crit}{Crit}
\def\TM+{T^*(\rr_+ \times M)}

\newcommand{\rr}{\ensuremath{\mathbb{R}}}

\theoremstyle{plain}
\newtheorem{thm}{Theorem}[section]
\newtheorem{cor}[thm]{Corollary}
\newtheorem{lem}[thm]{Lemma}

\newtheorem{prop}[thm]{Proposition}

\theoremstyle{definition}
\newtheorem{defn}[thm]{Definition}

\newtheorem*{ack}{Acknowledgments}

\theoremstyle{remark}
\newtheorem{rem}[thm]{Remark}

\numberwithin{equation}{section}
\newcommand{\norm}[1]{\left\lVert#1\right\rVert}
\newcommand{\dfn}[1]{{\textit {#1}}}
\newcommand{\leg}{\ensuremath{\Lambda}}

\newcommand{\rgh}[2]{\ensuremath{{GH}^{#1}(#2)}}
\newcommand{\lingf} {\mathcal F^{\operatorname{lin}}}

\newcommand{\Crit}{\operatorname{Crit}}

\begin{document}

\title[A product structure on Generating Family Cohomology]
{A product structure on Generating Family Cohomology for Legendrian Submanifolds}

\author[Z. Myer]{Ziva Myer} \address{Duke University, Durham, NC 27708} \email{zmyer@math.duke.edu} 

\begin{abstract} 
One way to obtain invariants of some Legendrian submanifolds in 1-jet spaces $J^1M$, equipped with the standard contact structure, is through the Morse theoretic technique of generating families.  This paper extends the invariant of generating family cohomology by giving it a product $\mu_2$. To define the product, moduli spaces of flow trees are constructed and shown to have the structure of a smooth manifold with corners. These spaces consist of intersecting half-infinite gradient trajectories of functions whose critical points correspond to Reeb chords of the Legendrian. This paper lays the foundation for an $A_\infty$ algebra which will show, in particular, that $\mu_2$ is associative and thus gives generating family cohomology a ring structure.
\end{abstract}

\date{\today}

\maketitle

\section{Introduction}
\label{sec:intro}
A classic contact $2n+1$-dimensional manifold is the 1-jet space of a smooth $n$-manifold $M$, $J^1M = T^*M \times \R$, equipped with the contact structure $\xi = \text{ker}(dz - \lambda)$, where $z$ is the coordinate on $\R$ and $\lambda$ is the Liouville one-form on $T^*M$.
An important  class of submanifolds in any $(2n+1)$-dimensional contact manifold are the Legendrian submanifolds:
 $n$-dimensional submanifolds $\Lambda$ such that $T_p\Lambda \subset \xi$ for all $p \in \Lambda$.
 Given a smooth function $f: M \to \R$, the 1-jet of $f$, $j^1f = \{(x, \frac{\partial{f}}{\partial{x}}, f(x))\} $, is a Legendrian submanifold of $J^1M$. 
Not all Legendrian submanifolds arise as the 1-jet of a function. Generating families are a Morse theoretical tool that encode a larger class of Legendrian submanifolds through
considering functions defined on a trivial vector bundle over $M$: 
 if a Legendrian $\Lambda \subset J^1M$ has a generating family $F:M \times \R^N \to \R$, then $\Lambda = \{ (x, \frac{\partial F}{\partial x}(x,e), F(x,e)) \} \mid 
 \frac{\partial F}{\partial e}(x,e) = 0\}$.  More background on generating families is given in Section~\ref{sec:back}.

In recent years, cohomology groups that are invariant under Legendrian isotopy have been defined for some Legendrian submanifolds in $J^1M$ through the  different techniques of
pseudoholomorphic curves and of generating families; see, for example, \cite{chv, ees:pxr, lisa:links, lisa-jill, f-r}.  In both of these constructions, the cohomology groups have an underlying cochain complex generated by the Reeb chords of the Legendrian.
 For a Legendrian $\Lambda$ in $J^1M$ with a generating family $F$, the Reeb chords are in bijective correspondence with the positive-valued critical points of a ``difference function" $w$ associated to $F$.  Thus by considering a cochain complex generated by the positive-valued critical points of $w$ and a coboundary map $\partial$ defined using the
 positive gradient flow of $w$, one can define generating family cohomology $GH^*(F)$.
Sometimes a Legendrian $\Lambda$ can have multiple, non-equivalent generating families: one then obtains an invariant of $\Lambda$
by considering  the set $\{ GH^*(F)\}$ for
all generating families of $\Lambda$.

Generating family cohomology  is an effective but not complete invariant, so a natural problem is to build further invariant algebraic structures on $GH^*(F)$. Towards this goal, we define a product structure:

\begin{thm}
Given a Legendrian $\Lambda \subset J^1M$ with a generating family $F: M \times \R^N \to \R$,  there exists a map
$$\mu_2: GH^i (F) \otimes GH^j (F) \rightarrow GH^{i+j}(F). $$
This map descends to the equivalence class $[F]$ with respect to the  operations of stabilization and fiber-preserving diffeomorphism of the generating family. 
Further, if $\Lambda^0$  has a generating family $F^0$ and $\Lambda^1$ is Legendrian isotopic to  $\Lambda^0$, then, letting $F^1$ be the associated generating family of $\Lambda^1$ and $\mu_2^i$ the associated product on $GH^*(F^i)$ for $i=0,1$, the following diagram commutes:
\begin{equation}
      \xymatrix{
       GH^i(F^0) \otimes GH^j(F^0) \ar[d]_{\cong}  \ar[r]^{ \quad\quad \mu_2^0} &GH^{i+j}(F^0) \ar[d]^{\cong} \\
       GH^i(F^1) \otimes GH^j(F^1) \ar[r]^{\quad \quad \mu_2^1} &GH^{i+j}(F^1),
      }
    \end{equation}
 where the vertical isomorphisms are induced by continuation maps constructed in Section \ref{sec:inviso}.
\end{thm}
This theorem appears in parts throughout this paper as Corollary \ref{cor:product}, Corollary \ref{cor:fpd}, Corollary \ref{cor:stab}, and Theorem \ref{inviso}.


Defining $\mu_2$ is part of a larger project in progress to define $A_{\infty}$ structures for Legendrian/Lagrangian submanifolds with generating families. This was inspired in part by Fukaya's $A_{\infty}$ category of Lagrangian submanifolds in a symplectic manifold, an extension of Floer homology, \cite{fooo}. In a toy model of Fukaya's construction, one gets an $A_\infty$ category extending the Morse cohomology of a manifold $M$ by studying gradient flow trees of Morse functions on $M$ \cite{fukaya1993morse} \cite{morse-htpy}; in Fukaya's full construction, gradient flow trees are replaced by pseudoholomorphic curves. Rather than using pseudoholomorphic curves to capture geometric information, our approach builds off of the toy model to build an $A_\infty$ category using gradient flow trees from generating families; this involves extending the tree construction from functions on $M$ to functions defined on trivial vector bundles over $M$. There are a number of analytic challenges in this approach,
including 
the fact that the geometric information is recorded in the subcomplex of the Morse cochain complex consisting of positive valued critical points and that standard generic perturbations of  functions  used for transversality arguments are no longer
possible since these perturbations destroy the  correspondence  of critical points with the geometric information of Reeb chords.

There are a number of interesting differences between  this generating family construction and analogous pseudoholomorphic curve constructions.
Pseudoholomorphic curve constructions have built a DGA \cite{yasha:icm, egh, ees:pxr}, whose homology is a strong invariant algebraic structure for Legendrian submanifolds of arbitrary dimensions, by using infinite-dimensional analysis of PDEs. Interestingly, Morse flow trees have been shown by Ekholm \cite{ekholm:morse-flow} to be a useful tool in calculating the DGA; these flow trees differ from the ones in this paper in several ways, including that they use local functions defined from the Legendrian rather global ones and can be used for a larger class of Legendrians than studied here. In low dimensions, combinatorial methods have been used to extract invariant algebraic structures similar to those that we are interested in from the DGA \cite{chv, lenny:computable, products, bc:bilinearized, ng2015augmentations}. 
Our work gives a different approach: it is a non-combinatorial, chain level construction for Legendrians with generating families in $J^1M$, where $M$ is  $\R^n$ or a closed $n$-manifold, for arbitrary $n$. Our approach differs from pseudoholomorphic curve constructions by only using finite-dimensional analytic techniques. 

 \begin{figure}[b]\label{ytreeone}
 \label{fig:Y-tree}
\includegraphics[width=1 \textwidth]{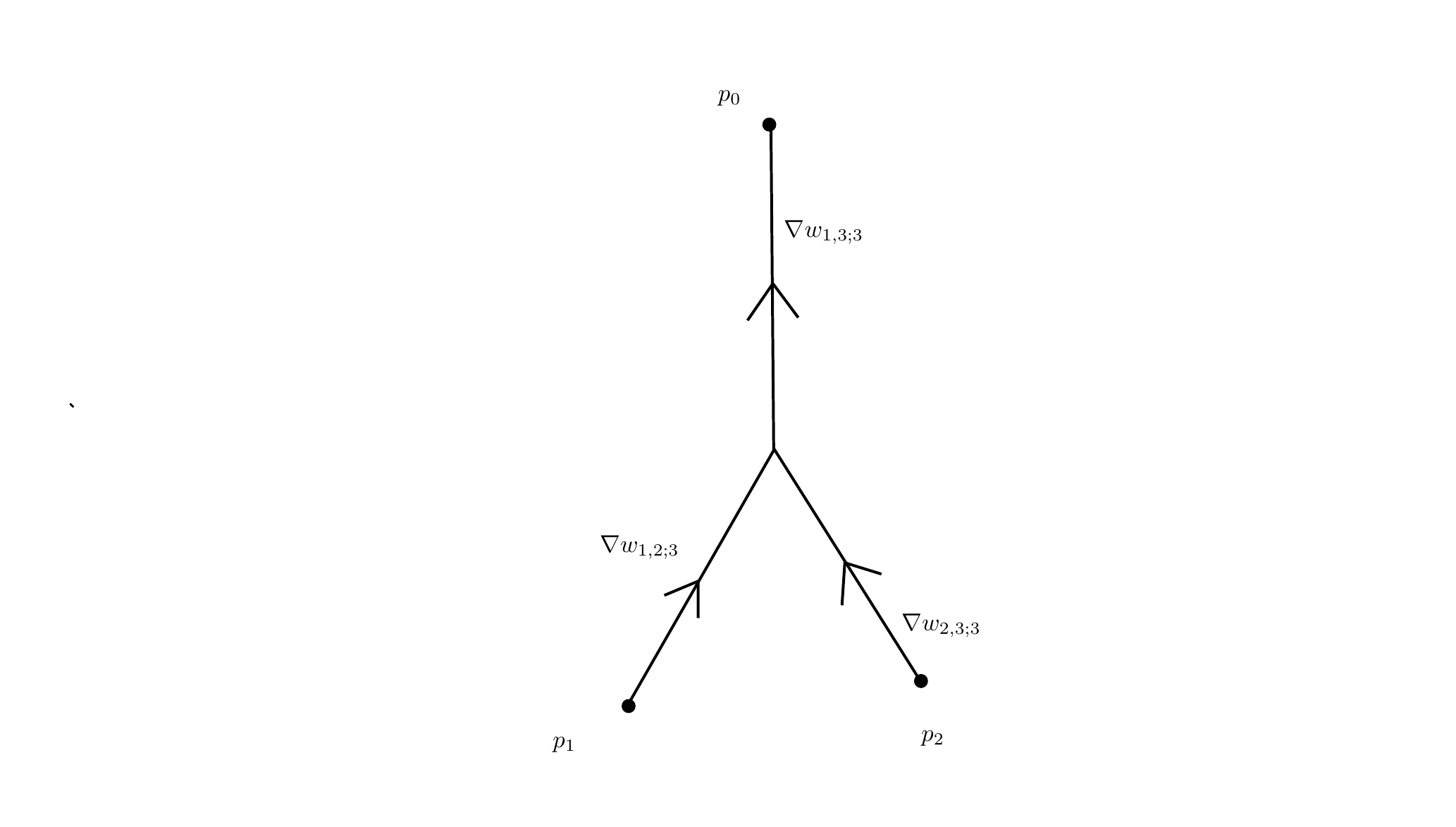} 
\caption{An element in $\mathcal{M}(p_1,p_2; p_0 | 0)$, i.e., with no perturbation, is a tree with three intersecting half-infinite trajectories that follow different quadratic stabilizations of the difference function $w$. A nonzero choice of $s$ produces a tree with ``almost" intersecting trajectories.}
\end{figure} 

The overview of the construction of the product $\mu_2$ and the layout of this paper is as follows: $\mu_2$ is defined through a count of points in a $0$-dimensional moduli space of gradient flow trees.  Namely, from $F$ (Section \ref{sec:back}), one  constructs the
difference function $w$ used to define generating family cohomology $GH^*(F)$ (Section \ref{sec:gh-legendrian}). The functions used in flow trees are three different ``quadratic-like" stabilizations of $w$, denoted by $\w{1}{2}{3}, \w{2}{3}{3}, \w{1}{3}{3}$ and defined in Section \ref{sec:extdiff}.  In Section \ref{sec:moduli}, for critical points $p_1, p_2, p_0$ of $w$, which correspond to
Reeb chords of $\Lambda$,  and a ``perturbation" parameter $s$, we construct a moduli space $\mathcal M(p_1, p_2; p_0 | s)$ of ``almost" intersecting gradient trajectories of  $\w{1}{2}{3}, \w{2}{3}{3}, \w{1}{3}{3}$; see Figure~\ref{fig:Y-tree}.
For generic choices of the $s$ parameter, $\mathcal M(p_1, p_2; p_0 | s)$ will be a manifold; 
for appropriate indices of $p_i$, $\mathcal M(p_1, p_2; p_0)$ will be $0$-dimensional and then $m_2$ is defined in Section \ref{sec:product} by a count of points in such a space.  Using the compactification of a 1-dimensional $\modsp$ from Section \ref{sec:moduli} shows that $m_2$ is a cochain map and thus descends to a map $\mu_2$ on $GH^*(F)$.
In Section \ref{secinvar}, we show that $\mu_2$ is invariant under equivalences of generating families from stabilization and fiber-preserving diffeomorphism. Lastly, we show in Section \ref{sec:inviso} that $\mu_2$ is invariant under Legendrian isotopy of $\Lambda$.
%

\begin{ack}
The author cannot thank Lisa Traynor, her graduate advisor, enough for guidance and countless conversations throughout this project. The author would also like to acknowledge Joshua Sabloff, Katrin Wehrheim, Dan Rutherford, and Lenny Ng for helpful discussions at various stages of this project. Lastly, the author thanks Bryn Mawr College for the supportive graduate school environment it provided to her while completing the majority of the research in this paper.
\end{ack}

\section{Background}
\label{sec:back}

In this section, we give some basic background on Morse Theory, generating families, and manifolds with corners.

\subsection{Morse Theory Basics}
To set notation, we recall a few facts from Morse Theory; see, for example, \cite{morse1934calculus, milnor:morse, schwarz, witten:morse, hutchings2002lecture} for more details. Let $X$ be closed manifold and let $f:X \to \R$ be a Morse function, i.e., a smooth function with nondegenerate critical points. We will relax the condition that $X$ is closed in future sections by using a function with taming properties outside a compact set. Given a critical point $p \in \crit(f)$, the \textit{Morse index} $\operatorname{ind}_f(p) \in \Z^{\geq0}$ is the dimension of the negative eigenspace of the Hessian $\operatorname{D}^2f(p)$. 

To study how a Morse function gives topological information, we pick an auxiliary Riemannian metric $g$ and study flow lines of $\nabla_gf$. For the purposes of this paper, we use \textit{positive} gradient flow. Let $\psi: \R \times X \to X$ denote the flow of this vector field and define the \textit{stable and unstable manifolds} for $p \in \crit(f)$ as
$$W_p^-(f) = \{ x \in X \mid \lim_{t \to -\infty} \psi_t(x) = p\} \quad \quad W_p^+(f) = \{ x \in X \mid \lim_{t \to \infty} \psi_t(x) = p\}. $$
These are smooth manifolds. Since we are using positive gradient flow, $W_p^-(f)$ is of dimension $\operatorname{coind}(p)$ while $W_p^+(f)$ is of dimension $\operatorname{ind}(p)$. The pair $(f,g)$ is called \textit{Morse-Smale} if $W_p^-(f) \pitchfork W_p^+(f)$, for all $p \in \Crit(f)$.

\subsection{Generating Family Background}
\label{ssec:gf-back}

In this subsection we review some background of generating families for Legendrian submanifolds that will be used in subsequent sections.  The main idea behind generating families stems from the fact that the 1-jet of a function $f: M \to \R$ is
a Legendrian submanifold of $J^1M$.  Generating families extend this idea to a larger class of Legendrian submanifolds by enlarging the domain of the function to the trivial vector bundle $M \times \R^N$ for some potentially large $N$ and considering the 1-jet over  a ``fiber-critical" submanifold of the domain.  We will denote the fiber coordinates on $\R^N$ by $e = (e_1, \ldots, e_N)$.  In this paper, $M$ will either be $\R^n$ or a closed $n$-manifold. This subsection contains only core definitions to set notation; for more information, see, for example, \cite{theret:viterbo, lisa:links, viterbo:generating}.

Given a function $F: M^n \times \R^N \to \R$ that is smooth and has $0$ as a regular value of $\partial_{e} F: M \times \R^N \to \R^N,$  define the \dfn{fiber critical submanifold} of $F$ to be the $n$-dimensional submanifold $\Sigma_F = (\partial_\eta F)^{-1}(0)$ and define the immersion $j_F: \Sigma_F \to J^1M$ in local coordinates by: 
$$
  j_F(x,e) = (x, \partial_x F(x,e), F(x,e)).
$$
The image $\leg$ of $j_F$ is an immersed Legendrian submanifold in $J^1M$ and for this setup we say that $F$ is a
\dfn{generating family} for $\leg$. 

\begin{defn}
Two generating families $F_i: M \times \R^{N_i} \to
\R$, $i= 0, 1$, are said to be \dfn{equivalent}, denoted $F_0 \sim F_1$, if they can be made equal through applying fiber-preserving diffeomorphisms and
stabilizations. These operations are defined as follows:
\begin{itemize}
\item Given $F: M \times \R^N \to \R$ and a  nondegenerate quadratic function $Q:\R^K \to \R$, define $F
  \oplus Q: M \times \R^N \times \R^K \to \R$ by $F \oplus Q(x,
  e, e') = F(x, e) + Q(e')$. Then  we say $F \oplus Q$ is a \dfn{stabilization} of $F$ (of dimension $K$).
\item For $F: M \times \R^N \to \R$, a fiber-preserving
  diffeomorphism is a map $\Phi: M \times \R^N \to M \times \R^N$ such that $\Phi(x,e) = (x, \phi_x(e))$ where $\phi_x$ is a smooth family of 
  diffeomorphisms on $\R^N$. Then $F \circ \Phi$ is said to be
  obtained from $F$ by precomposition with a \dfn{fiber-preserving diffeomorphism}.
\end{itemize}
Given a generating family $F$, we use $[F]$ to denote the equivalence class of $F$ with
respect to stabilization and fiber-preserving diffeomorphism.
\end{defn}

Using the above definitions, it is not hard to show that if $F: M \times \R^N \to \R$ is a generating family for a Legendrian $\leg$, then any $\widehat{F} \in [F]$ will also be a generating family for $\leg$. If a Legendrian submanifold has a generating family, it will always have an infinite number of generating families, but the set of equivalence classes may be finite. See the beginning of Subsection \ref{ssec:rghLeginv} for a brief discussion of results on this topic.

Having a generating family to work with will allow us to use concepts from Morse homology, as explained in the following section. As the domain of our functions are non-compact, we impose the following ``tameness" property on our generating families:

\begin{defn} \label{defn:li} A function $f: M \times \R^N \to \R$ is said to be \dfn{linear-at-infinity} if $f$ can be expressed as
$$f(x, e) = f^c(x, e) + A(e),$$
for a function $f^c:M \times \R^N \to \R$ with compact support and a non-zero linear function $A: \R^N \to \R$. 
\end{defn}

\begin{rem}\label{namingcompact}
As $M$ is a closed manifold or $\R^n$, we may assume that the compact set in Definition \ref{defn:li} is of the form $K_M \times K_E \subset M \times \R^N$ where $K_M = M$ if $M$ is closed or $K_M$ and  $K_E$ are compact Euclidean subsets.
\end{rem}

The linear-at-infinity condition is useful for studying compact Legendrians when $M= \R^n$, such as in \cite{f-r, lisa-jill}. The definition of linear-at-infinity is not preserved under stabilization of the generating family. However, we have:

\begin{lem}[\cite{josh-lisa:obstr}] \label{lem:lq-l} If $F$ is a stabilization of a
  linear-at-infinity generating family, then $F$ is equivalent to a
  linear-at-infinity generating family.
\end{lem}

\subsection{Smooth Manifolds with Corners}

Many proofs of theorems in the Section \ref{sec:moduli} use differential topology applied to smooth manifolds with corners, which we review in this section. 

\begin{defn}
A \dfn{smooth manifold with corners} of dimension $n \in \mathbb{N}$ is a second-countable, Hausdorff space $X$ equipped with a maximal atlas of charts $\left\lbrace \phi :X \supset U \to V \subset [0, \infty )^n \right\rbrace  $ whose transition maps are smooth. The \dfn{$\ell$-stratum $X_{\ell}$} is the set of points $x \in X$ such that $\phi(x)$ has $\ell$ components equal to 0.
\end{defn}

\begin{lem}\label{cornerproducts}
If $X$ and $Y$ are smooth manifolds with corners of dimensions $m_1$ and $m_2$, respectively, then $X \times Y$ is a smooth manifold with corners of dimension $m_1 + m_2$ and $$(X \times Y)_i = \bigsqcup_{j+k=i} X_j \times Y_k.$$
\end{lem}

%

There is a natural way to understand smooth maps and derivatives on manifolds with corners. For our purposes, understanding maps from a smooth manifold with corners to a smooth manifold without boundary or corners will suffice. Consider a smooth map $f: U \subset [0, \infty )^n \to \R^\ell$ for some $n, \ell$. If $u \in U$ has no coordinates equal to 0, then $df_u$ is our usual notion of derivative. However, if $u$ has some 0 coordinates, the smoothness of $f$ implies that we may extend $f$ to a smooth map $\tilde{f}$ on a neighborhood of $u$ in $\R^n$. We define $df_u$ to be the usual derivative $d \tilde{f}_u: \R^n \to \R^\ell$. This will not depend on local extension of $f$. 

With this observation, given a manifold with corners $X$, we may define the tangent space $T_xX$ at $x \in X$ to be the image of the derivative of any local parametrization about $x$. Given a map defined on a manifold with corners $X$, let $\partial_i f: X_i \to Y$ denoted the restriction of $f$ to the stratum $X_i$. Then $T_xX_i$ is a linear subspace of $T_xX$ of codimension $i$ and $d(\partial_if)_x = df_x|_{T_xX_i} $. 

We will make use of the following natural extension of classic differential topology theorems of manifolds with boundary to manifolds with corners. We must impose extra transversality conditions on the strata to achieve these results.

The following is a known result (see, for example, \cite{nielsen1982transversality}). 

\begin{thm}\label{preimage}
Let $f$ be a smooth map of a manifold $X$ with corners onto a boundaryless manifold $Y$, and suppose $f: X \to Y$ and $\partial_i f: X_i \to Y$ are transversal to a boundaryless submanifold $Z \subset Y$ for all strata $X_i$ of $X$. Then the preimage $f^{-1}(Z)$ is a manifold with corners with $i$-stratum $f^{-1}(Z)_i = f^{-1}(Z) \cap X_i$ and the codimension of $f^{-1}(Z)$ in $X$ equals the codimension of $Z$ in $Y$.
\end{thm}

The following is a generalization of the Transversality Theorem for manifolds with boundary (see, for example Section $2.3$ in \cite{g-p}) to manifolds with corners.

\begin{thm}\label{transcorners}
Suppose $F:X \times S \to Y$ is a a smooth map, where $X$ is a manifold with corners and $S$ and $Y$ are boundaryless manifolds. Let $Z$ be a boundaryless submanifold of $Y$. Suppose $F:X \times S \to Y$ is transversal to Z and $\partial_i F :X_i \times S \to Y$ is transversal to $Z$ for all i-strata $X_i$ of $X$. Then for almost every $s\in S$, $f_s$ and $\partial_i f_s$ are transversal to $Z$ for each $i$-stratum $X_i$, where $f_s(x) = F(x,s)$ and $\partial_i f_s = f_s|_{X_i}$.
\end{thm}

\section{A Flow Line Approach to Generating Family Cohomology}
\label{sec:gh-legendrian}

In this section we use gradient flow lines of a certain function (called a difference function and defined below in Equation (\ref{diff})) associated to a generating family in order to define generating family cohomological invariants of Legendrians. This setup differs from past formulations of (co)homology for generating families defined using the singular (co)homology of sublevel sets as in \cite{f-r, lisa-jill, lisa:links}, for example.

\subsection{Setup of $GH^*(F)$}

Given a generating family $F: M \times \R^N \to \R$ for a Legendrian $\leg \subset J^1M$, the \dfn{ difference function} of $F$, $w: M \times \R^N \times \R^N \to \R$, is defined by: \begin{equation}\label{diff} w(x,e,e') = F(x,e) - F(x,e').  \end{equation}
 
The difference function is important because its critical points
contain information about the \dfn{Reeb chords} of $\leg$ (with respect to the standard contact form on $J^1M$), which for our purposes are segments $\gamma: [a,b] \to J^1M$ parallel to the positive $z$-axis whose endpoints are on $\leg$. 
  For the following proposition, let $\ell(\gamma) > 0$ be the length of the Reeb chord $\gamma$, as measured along the $z$-axis.

\begin{prop}[\cite{f-r, josh-lisa:cap}] 
  \label{prop:leg-crit-point} 
  The difference function $w$ has two types of critical points:
  \begin{enumerate}
  \item\label{isolatedcrits} For each Reeb chord $\gamma$ of $\leg$, there exist exactly two
    critical points $(x,e,e')$ and $(x,e',e)$ of $w$ with
    nonzero critical values $\pm \ell(\gamma)$.
  \item The set
    $$\left\{ (x,e,e) \, \mid \, (x,e) \in \Sigma_F\right\}$$
    is a critical submanifold of $w$ and has critical value $0$.
  \end{enumerate}
  For generic $F$, these critical points and submanifold are
  nondegenerate, and the critical submanifold is of index $N$.
\end{prop}

In this paper, we will work with the critical points of $w$ of Type (\ref{isolatedcrits}) that have positive critical value. 


\begin{defn}\label{gradedvs}
Given $F: M \times \R^N \to \R$ and associated difference function $w: M \times \R^N \times \R^N \to \R$, let $\crit_+(w)$ be the set of critical points of $w$ with positive critical value. Then define $\V{} \coloneqq \left\langle \crit_+(w)\right\rangle_{\Z_2}$ to be the vector space generated over $\Z_2$  by elements in $\crit_+(w)$. Equip $\V{}$ with the following grading on generators:
$$|p| = \ind w(p) - N.$$
\end{defn}

\begin{rem}
 The shift in index occurs so that the groups are invariant when $F$ undergoes a stabilization operation. Note that some previous formulations of $GH^*(F)$ use a shift of $N+1$ rather than $N$ to produce an isomorphism with linearized contact homology; see \cite{f-r}. We use a shift of $N$, however, to guarantee that our product map has the standard degree.
\end{rem}

The positive-valued critical points, and in fact all critical points of $w$, are contained in a compact subset of $M \times \R^{2N}$.

\begin{lem}\label{wcritscmpct} 
Suppose that $F: M \times \R^N \to \R$ is a linear-at-infinity generating family that agrees with a non-zero linear
function outside $K_M \times K_E$, for compact sets $K_M \subset M$ and $K_E \subset \R^N$ as in Remark \ref{namingcompact}. Then every critical point of the
associated difference function $w: M \times \R^N \times \R^N \to \R$ is contained in $K_M \times K_E \times K_E$.
\end{lem}

\begin{proof}  First consider the critical points of $w$.  By assumption 
$$F(x, e) = F^c(x, e) + A(e),$$
where $F^c = 0$ if $x \notin K_M$ or $x \notin K_E$, and $A$ is a nonzero linear function.  
Thus we see that for all $(x, e_1, e_2) \in M \times \R^N \times \R^N$,
$$w(x, e_1, e_2) = F^c(x, e_1) - F^c(x, e_2) + A(e_1) - A(e_2).$$
We want to show that if $(x, e_1, e_2)$ is a critical point of $w$, then $x \in K_M$, $e_1 \in K_E$, and $e_2 \in K_E$.
Suppose for a contradiction that $x \notin K_M$.  Then we see that $w(x, e_1, e_2)$ agrees with the linear
function $A(e_1) - A(e_2)$, and thus $(x, e_1, e_2)$
cannot be a critical point.  If  $e_1 \notin K_E$, $w(x, e_1, e_2) =  - F_0(x, e_2) + A(e_1) - A(e_2)$, and thus
the $\frac{\partial w}{\partial e_1}(x, e_1, e_2) \neq 0$, showing that $(x, e_1, e_2)$ is not critical for $w$.  A similar argument
shows that if  $e_2 \notin K_E$,  $\frac{\partial w}{\partial e_2}(x, e_1, e_2) \neq 0.$  Thus if $(x, e_1, e_2)$ is critical for $w$,
then $x \in K_M$, $e_1 \in K_E$, and $e_2 \in K_E$. \qedhere
\end{proof}

Just as the condition that $F$ is linear-at-infinity is not preserved under stabilization of $F$, if $F$ is a linear-at-infinity generating family, then the associated difference function $w$ is no longer linear-at-infinity.   However, we have:

\begin{lem}[\cite{f-r}] \label{lem:lin-diff} If $F$ is a linear-at-infinity generating family, then the associated difference function $w$ is equivalent to a linear-at-infinity function.  \end{lem}

Since Reeb chords of a Legendrian with a generating family $F$ are in bijection with positive-valued critical points of the difference function $w$ of $F$, the idea behind generating family cohomology is to study the Morse cohomology of the set $\{w>0\}$. 
To do this, we first equip the domain of the $w$ with a Riemannian metric.


\begin{defn}\label{wEMS}
Let $F: M \times \R^N \to \R$ be a linear-at-infinity generating family that agrees with a non-zero linear function outside $K_M \times K_E$, for compact sets $K_M \subset M$ and $K_E \subset \R^N$ as in Remark \ref{namingcompact}. Let $F$ have difference function $w: M \times \R^N \times \R^N \to \R$. Let the set of \textbf{compatible metrics}, $\mathcal{G}_F$, denote the set of Riemannian metrics $g_w$ on $M\times \R^N \times \R^N$ so that for all $p \in \Crit_+{w}$, there is a neighborhood $U$ of $p$ and a parametrization $\phi: B^{\ind{p}} \times B^{(n+2N) - \ind{p}} \to U$ with $\phi(0) =p$ such that
\begin{enumerate}
\item\label{Morse}  $\phi^*w = w(p) + \frac{1}{2}(x_1^2 + \dots + x_{(n+2N) - \ind{p}}^2) - \frac{1}{2}(x_{(n+2N) - \ind{p} +1}^2 + \dots + x_{n+2N}^2), $
\item\label{Euclidean} $\phi^*g_w = \text{d}x_1 \otimes \text{d} x_1 + \dots + \text{d}x_{n+2N} \otimes \text{d}x_{n+2N}. $ 
\end{enumerate}
In addition, metrics in $\mathcal{G}_F$ will satisfy
\begin{enumerate}[resume]
\item\label{stdatinf} Outside $K_M \times K_E \times K_E$, $g_w$ is the standard Euclidean metric, and
\item\label{Smale} For every pair of critical points $p,q \in \Crit_+{w}$, the unstable and stable manifolds of $p$ and $q$ have a transverse intersection.
\end{enumerate}
\end{defn}

\begin{rem}
\begin{enumerate}
\item Conditions (\ref{Morse}), (\ref{Euclidean}), and (\ref{stdatinf}) are standard Morse theoretic conditions so that we may understand gradient flow near critical points and outside the compact set. In particular, these assumptions allow us to use results of \cite{wehrheim:morse} in Section \ref{sec:moduli}. While condition (\ref{Euclidean}) is not generic, the gradient flow of a pair satisfying conditions (\ref{Morse}), (\ref{stdatinf}), and (\ref{Smale}) is topologically conjugate to one satisfying all four; see \cite[Remark 3.6]{wehrheim:morse} or \cite{franks1979morse}.
\item It is possible to find metrics satisfying (\ref{Morse}) - (\ref{Smale}). Start with a metric satisfying (\ref{Morse}) - (\ref{stdatinf}) and then perform $L^2$-small perturbations of the metric on annuli around critical points to get the additional Smale condition (\ref{Smale}), see \cite{bhsmale}.
\end{enumerate}
\end{rem}

%
%

To compute cohomology groups on this graded vector space, we will equip $\V{}$ with a codifferential $\delta: \V{*} \to \V{*+1}$ defined by a count of isolated gradient flow lines, modulo reparametrization in time. 

In particular, if $p, q \in \V{}$ and $g \in \mathcal{G}_F$, let
\begin{align*}
&\M(p,q) \coloneqq \\
&\left( \{ \gamma:\R \to M \times \R^N \times \R^N \mid \dot{\gamma} = \nabla_gw, \lim_{t \to - \infty} \gamma(t) = p, \lim_{t \to \infty} \gamma(t) = q\} \right) / \R,
\end{align*} 
where $\R$ denotes the action of translation in the $t$ variable.

\begin{thm}\label{mfldstr}
$\M(p,q)$ is a smooth manifold of dimension $|q|-|p|-1$.
\end{thm}

Theorem \ref{mfldstr} follows from the natural identification 
$$\M(p,q) \cong  W^-_p(w) \cap W^+_q(w) \cap w^{-1}(c),$$
for $c \in \R$ a regular value in $(w(p), w(q))$. The Morse-Smale assumption (Condition \ref{Smale} in Definition \ref{wEMS}) guarantees that this is a smooth manifold of the above dimension. 

\begin{defn}
Define the map $\delta : \V{*} \to \V{*+1}$ by
$$\delta(p) = \sum_{q \in \V{|p|+1}}  \#_{\Z_2} \M(p,q) \cdot q.$$
\end{defn}

\begin{rem}
The map $\delta$ is well defined: while $\V{}$ is generated by only positive-valued critical points of $w$, $\delta$ counts flow lines of the positive gradient flow, so $w$ will increase in value along trajectories. 
\end{rem}

In fact, $\delta$ is a codifferential:
\begin{lem}
The map $\delta : \V{*} \to \V{*+1}$ satisfies $\delta^2 = 0$. 
\end{lem}

\begin{proof}
Given our taming condition of requiring $F$ to be linear-at-infinity, the relevant flow lines are contained in a compact set; see Lemma \ref{wcritscmpct} and note that outside this compact set, the gradient flow is partially constant. Thus, the above is shown by the standard argument for functions with a compact domain-- If $\M(p,q)$ is a 1-dimensional manifold, it may be compactified with the addition of once-broken flow-lines, which make up its boundary. As the boundary of a 1-dimensional manifold contains an even number of points, the result follows. For more details, see, for example, \cite{schwarz}.
\end{proof}


\begin{defn} 
  \label{defn:gh}  
  The 
  \dfn{generating family cohomology} $\rgh{*}{F}$ 
  of the  generating family $F$ 
  is defined to be:
$$\rgh{*}{F} = H^*(\V{}, \delta). $$
\end{defn}

\begin{rem}
Note that we built the usual grading shift into the definition of the index of $\V{}$ rather than into the definition of the cohomology as was done in past papers on $GH^*(F)$, such as in \cite{f-r, lisa-jill, lisa:links}.
\end{rem}

\subsection{Independence of $\rgh{*}{F}$ with respect to metric choice}

Crucial to the above construction of $\rgh{*}{F}$ is the metric $g \in \mathcal{G}_F$ used to produce a gradient flow of $w$. In this subsection, we show that $\rgh{*}{F}$ is independent of the metric used in its construction. While different metrics affect the codifferential $\delta$ of $\rgh{*}{F}$, a continuation argument shows that the cohomology is invariant up to isomorphism under generic change of metric. Continuation arguments will be used multiple times in this paper, so we provide a detailed exposition for the following proposition:

\begin{prop}\label{metricinvrgh}
Up to isomorphism, $\rgh{*}{F}$ does not depend on the metric $g \in \mathcal{G}_F$ used in the construction.
\end{prop}

\begin{proof}
To show that $\rgh{*}{F}$ does not depend on the metric, we use the idea of Floer homology's continuation maps from \cite{floer:arnold}, \cite{hutchings:families}; see also \cite{hutchings2002lecture} for an informal exposition. This technique will be used multiple times to show different notions of invariance; as this is the first use, we will provide the details here for later reference.

The strategy of the proof is to define a continuation cochain map $\Phi_\Gamma$ from a path of functions and metrics $\Gamma$ (Lemma \ref{contcochain}). Given a homotopy between two such paths $\Gamma$ and $\widehat{\Gamma}$, we construct a chain homotopy $K$ between the two continuation maps of the two paths in Lemma \ref{pathhtpy}. We then show that in Lemma \ref{concathtpy} that the continuation map of the concatenation of any two paths $\Phi_{\Gamma_2*\Gamma_1}$ is chain homotopic to the composition of the two continuation maps from the paths $\Phi_{\Gamma_2} \circ \Phi_{\Gamma_1}$. Last, we show that the continuation map of a constant path of a Morse-Smale pair is the identity map (Lemma \ref{constpath}). With all of these pieces, suppose that we have an arbitrary admissible path of functions and metrics. Concatenating the path with its reverse is homotopic to the constant path, showing that the continuation map is an isomorphism on cohomology.

To construct a continuation map, we use a path $\Gamma = \{(w^t,g^t) \mid t \in \left[ 0,1 \right] \}$ of difference functions and metrics. For this proof, we may set $w^t= w$ for all $t \in [0,1]$. Given two metrics $g^0, g^1 \in \mathcal{G}_F$, construct a path $g^t$ in the space of Riemannian metrics on $M \times \R^{2N}$ that are standard outside $K_M\times K_E \times K_E$. This space is contractible because we can use a straight-line homotopy of the metrics on the non-standard part to contract to any given metric in this set.




Given $\Gamma$, we construct a continuation map which we will denote by $\Phi_\Gamma$, with
$$\Phi_\Gamma: C^*(F^0) \to C^*(F^1),$$
where (in this case) $F^0 = F^1 =F$, the generating family that produced $w$. Given $\epsilon > 0$ such that $\dfrac{\epsilon}{4} < \rho$, where $\rho$ is the least positive critical value of $w$, the continuation maps count isolated gradient flow lines of the vector field $\nabla_GW$ on $(M \times \R^{2N}) \times I$ with 
\begin{align}\label{nabGW}
\begin{split}
W(p,t) &= w^t(p) + \epsilon \left((1/2)t^2 - (1/4)t^4\right) \\
G_{(p,t)} &= g^t_p + dt^2.
\end{split}
\end{align}

We say that the path $\Gamma$ is \textbf{admissible} if the unstable and stable manifolds of $\nabla_GW$ intersect transversely. Note that this does not mean that each $(w^t, g^t)$ is Morse-Smale.
%

This vector field has the following property: when projected to $I$, there's a critical point of index 0 at $t=0$ and one of index 1 at $t=1$ with none in between. The vector field flows smoothly from 0 to 1. Let $\crit_+^k(W)$ denote the set of critical points of $W$ with critical value greater than $\dfrac{\epsilon}{4}$. Then we have that
$$\crit_+^k(W) = \crit_+^k(w^0) \times \{0\} \bigcup \crit_+^{k-1}(w^1) \times \{1\}, $$
where the superscript denotes the Morse index. 


For $p \in \crit_+(w^0)$ and $q \in \crit_+(w^1)$, consider the space $\M_\Gamma((p,0),(q,1))$ of flow lines of $\nabla_GW$ from $(p,0)$ to $(q,1)$, modulo reparametrization. If $\Gamma$ is admissible, usual Morse theoretic arguments will show that $\M_\Gamma((p,0),(q,1))$ is a manifold of dimension $\ind_W(q,1) - \ind_W(p,0) -1 = \ind_{w^1}(q)+1 - \ind_{w^0}(p)+1 = \ind_{w^1}(q) - \ind_{w^0}(p)$.

Thus, we can define the map $\Phi_\Gamma$ on generators in $C^*(F^0)$ by
$$\Phi_\Gamma(p) =   \sum_{q \in C^{|p|}(F^1)}  \#_{\Z_2} \M_\Gamma((p,0),(q,1)) \cdot q.$$

We wish to show that the continuation map $\Phi_\Gamma$ gives an isomorphism on $GH^*(F)$. To show this, we must prove that $\Phi_\Gamma$ is a cochain map, that a homotopy of $\Gamma$ induces a chain homotopy, that concatenating paths gives a chain homotopy, and that the constant path gives the identity. These facts together show that our continuation map will induce an isomorphism on cohomology \cite{hutchings:families}.

%

\begin{lem}\label{contcochain}
For the path $\Gamma$, $\Phi_\Gamma$ is a cochain map, i.e., the following diagram commutes:
\begin{figure}[htb] 
\begin{tikzcd} 
C^k(F^0)  \arrow{r}{\Phi_\Gamma} \arrow{d}{\delta^0}
&C^k(F^1) \arrow{d}{\delta^1}\\
C^{k+1}(F^0) \arrow{r}{\Phi_\Gamma} &C^{k+1}(F^1)
\end{tikzcd}
\end{figure} 

Thus, $\Phi_\Gamma$ descends to cohomology.
\end{lem}

\begin{proof}
Consider a 1-dimensional moduli space $\mathcal{M}_\Gamma((p,0),(q,1))$ of flow lines along the vector field $\nabla_GW$ defined above (\ref{nabGW}) for $(p,0) \in \crit_+(W)$ to $(q,1) \in \crit_+(W)$. Since this space is one-dimensional, we have that $\ind_W(q,1) - \ind_W(p,0) = 2$, i.e., $\ind_w(q) - \ind_w(p) =1$. If $\Gamma$ is admissible, the usual Morse Theory compactification argument implies that $\mathcal{M}_\Gamma((p,0),(q,1))$ has a compactification to a compact 1-manifold with boundary  consisting of once-broken flow lines. Since these flow lines may only break at $t=0$ or $t=1$, we have the following expression for the boundary:
\begin{align*}
\partial \overline{\mathcal{M}_\Gamma}((p,0),(q,1)) &= \bigcup_{q' \in \text{Crit}^{\ind(p)}_+(w^1)} \mathcal{M}_\Gamma((p,0),(q',1)) \times \mathcal{M}_\Gamma((q',1),(q,1)) \\
&\cup \bigcup_{p' \in \crit^{\ind(p)+1}_+(w^0)} \mathcal{M}_\Gamma((p,0),(p',0)) \times \mathcal{M}_\Gamma((p',0),(q,1)).
\end{align*}
Thus, 
\begin{align*}
0 &= \sum_{q' \in \text{Crit}^{\ind(p)}_+(w^1)} \#_{\Z_2}\mathcal{M}_\Gamma((p,0),(q',1)) \cdot \#_{\Z_2} \mathcal{M}_\Gamma((q',1),(q,1)) \\
&+ \sum_{p' \in \crit^{\ind(p)+1}_+(w^0)} \#_{\Z_2}\mathcal{M}_\Gamma((p,0),(p',0)) \cdot \#_{\Z_2} \mathcal{M}_\Gamma((p',0),(q,1))
\end{align*}
The dynamics of $\nabla_GW$ imply that flows between critical points at fixed time $t=0$ or $t=1$ are completely contained in $M \times \R^{2N} \times \{t\}$. Thus, we have the following natural identifications of the following manifolds:
$$\mathcal{M}_\Gamma((q',1),(q,1)) = \mathcal{M}_{w^1}(q',q)$$
$$\mathcal{M}_\Gamma((p,0),(p',0)) = \mathcal{M}_{w^0}(p,p').$$
So we have that 
\begin{align*}
0 &= \sum_{q' \in \text{Crit}^{\ind(p)}_+(w^1)} \#_{\Z_2}\mathcal{M}_\Gamma((p,0),(q',1)) \cdot \#_{\Z_2} \mathcal{M}_{w^1}(q',q) \\
&+ \sum_{p' \in \crit^{\ind(p)+1}_+(w^0)} \#_{\Z_2} \mathcal{M}_{w^0}(p,p') \cdot \#_{\Z_2} \mathcal{M}_\Gamma((p',0),(q,1)).
\end{align*}
The result that $0 = \delta^1 \circ \Phi_\Gamma + \Phi_\Gamma  \circ \delta^0$ follows.
\end{proof}


The construction so far relied heavily on the path $\Gamma$ and it is necessary to show that $\Phi_\Gamma^*$, the induced map on cohomology, does not depend on the path chosen in a given homotopy class of paths. For the current proof we will take a homotopy of the path of metrics on $M\times \R^{2N}$ that are standard outside $K_M \times K_E \times K_E$. Since this space is contractible, the induced isomorphism will be canonical, showing that $\rgh{*}{F}$ is independent of the metric chosen in the construction. 

Choose another path $(\widehat{g}^t)$ with $g^0 = \widehat{g}^0$ and $g^1 = \widehat{g}^1$ (i.e., change the path but not the endpoints), and suppose there is a generic homotopy between these two paths. We wish to say that the corresponding continuation maps are chain homotopic (see Figure \ref{cochainhtpydiagram}). 

\begin{lem}\label{pathhtpy}
Given admissible paths $\Gamma$ and $\widehat{\Gamma}$ from $(w^0,g^0)$ to $(w^1,g^1)$, a fixed endpoint homotopy from the path $\Gamma$ to $\widehat{\Gamma}$ induces a chain homotopy
$$K: C^*(F^0) \to C^{*-1}(F^1)$$
between the maps $\Phi_\Gamma$ and $\Phi_{\widehat{\Gamma}}$; see Figure \ref{cochainhtpydiagram}.

\end{lem}

\begin{figure}
\begin{tikzpicture}[baseline= (a).base]
\node[scale=1.2] (a) at (0,0){
\begin{tikzcd}
\cdots \arrow[r, "\delta^0"] 
&C^{k-1}(F^0) \arrow[r, "\delta^0"] \arrow[ld, "K^{k-1}"' near start] \arrow[d, shift left, "\Phi_{\Gamma}"] \arrow[d, shift right, "\Phi_{\widehat{\Gamma}}"']
&C^{k}(F^0) \arrow[r, "\delta^0"] \arrow[ld, "K^k"' near start] \arrow[d, shift left, "\Phi_{\Gamma}"] \arrow[d, shift right, "\Phi_{\widehat{\Gamma}}"']
&C^{k+1}(F^0) \arrow[r, "\delta^0"] \arrow[ld, "K^{k+1}"' near start] \arrow[d, shift left, "\Phi_{\Gamma}"] \arrow[d, shift right, "\Phi_{\widehat{\Gamma}}"']
&\cdots \\
\cdots \arrow[r, "\delta^1"]
&C^{k-1}(F^1) \arrow[r, "\delta^1"]
&C^k(F^1) \arrow[r, "\delta^1"]
&C^{k+1}(F^1) \arrow[r, "\delta^1"]
&\cdots
\end{tikzcd}
};
\end{tikzpicture}

\caption{A chain homotopy $K$ between $\Phi_{\Gamma}$ and $\Phi_{\widehat{\Gamma}}$ is a sequence of maps that makes the above diagram commute.}
\label{cochainhtpydiagram}
\end{figure}

\begin{proof}
The image of the path homotopy between $\Gamma$ and $\widehat{\Gamma}$ traces out the shape of a digon $D$, a smooth two-dimensional manifold with corners consisting of two vertices, two edges in between them, and one face. The vertices correspond to fixed endpoints of the paths in the homotopy while the edges correspond to the two homotopic paths. For every $d \in D$, the homotopy gives a pair $(w^d,g^d)$, where, in this case, $w^d=w$ and $g^d$ is a metric on $M \times \R^{2N}$ that is standard outside the nonlinear-support compact set $K_M \times K_E \times K_E$.

Let $h$ be a metric on $D$ such that the edges of $D$ have length one. Let $f:D \to \R$ be a nonnegative function on the digon that has an index 0 critical point at one vertex $d_0$ with critical value 0, an index 2 critical point at the other vertex $d_1$ with critical value $\dfrac{\epsilon}{4}$ for $\epsilon > 0$ as in the equation (\ref{nabGW}) of $\nabla_GW$, and no other critical points. Lastly suppose $\nabla_hf$ is tangent to the edges of the digon and agrees with the standard gradient of $\epsilon \left((1/2)t^2 - (1/4)t^4\right)$ on each edge.

To get a chain homotopy, we will consider certain flow lines of the function $W^D: (M \times \R^{2N}) \times D \to \R$ and metric $G^D$ defined by
\begin{align*}
W^D(p,d) &= w^d(p) + f(d) \\ 
G^D_{(p,d)} &= g^d_p + h_d
\end{align*}

Denote the critical points of $W^D$ with critical value greater than $\dfrac{\epsilon}{4}$ by $\crit_+(W^D)$. Since the only critical points of the digon occur at the vertices $d_0$ and $d_1$,

$$\crit_+^k(W^D) = \crit_+^k(w^0) \times \{d_0\} \bigcup \crit_+^{k-2}(w^1) \times \{d_1\}. $$

The homotopy is admissible if the stable and unstable manifolds of $\nabla_{G^D}W^D$ have a transverse intersection. Given an admissible homotopy, the space $\M_D((p,d_0),(q,d_1))$ of gradient flow lines modulo reparametrization is a manifold of dimension 
$$\ind_{W^D}(q,1) - \ind_{W^D}(p,0) -1 = \ind_{w^1}(q)+2 - \ind_{w^0}(p)+1 = \ind_{w^1}(q) - \ind_{w^0}(p) +1.$$
Thus, we may define a map $K_D: C^*(F^0) \to C^{*-1}(F^1)$ by
$$K_D(p) =   \sum_{q \in C^{|p|-1}(F^1)}  \#_{\Z_2} \M_D((p,d_0),(q,d_1)) \cdot q.$$

To show that $K_D$ gives a chain homotopy between $\Phi_\Gamma$ and $\Phi_{\widehat{\Gamma}}$, we use the usual argument that if $\mathcal{M_D}((p,0),(q,1))$ is one-dimensional, then it has a compactification to a compact one-dimensional manifold with boundary. The boundary contains the usual once-broken flow lines and has additional flow lines from the boundary of the digon, $\partial D$.

\begin{align*}
\partial \overline{\mathcal{M_D}}((p,0),(q,1)) &= \bigcup_{q' \in \text{Crit}^{\ind(p)-1}_+(w^1)} \mathcal{M}((p,0),(q',1)) \times \mathcal{M}((q',1),(q,1)) \\
&\cup \bigcup_{p' \in \crit^{\ind(p)+1}_+(w^0)} \mathcal{M}((p,0),(p',0)) \times \mathcal{M}((p',0),(q,1))\\
&\cup \M_\Gamma((p,0),(q,1)) \; \cup \; \M_{\widehat{\Gamma}}((p,0),(q,1))
\end{align*}

Thus, $K \delta^0 + \delta^1 K = \Phi_{\Gamma} - \Phi_{\widehat{\Gamma}}$. 
\end{proof}

\begin{lem}\label{concathtpy}
Given admissible paths $\Gamma_1, \Gamma_2$ with $\Gamma_1(1)=\Gamma_2(0)$, there is a concatenation chain homotopy between $\Phi_{\Gamma_2} \circ \Phi_{\Gamma_1}$ and $\Phi_{\Gamma_2 *\Gamma_1}$.
\end{lem}

\begin{proof}
This proof is similar to the proof of Lemma \ref{pathhtpy}. An admissible homotopy between $\Gamma_1$ followed by $\Gamma_2$ with their concatenation $\Gamma_2*\Gamma_1$ may be represented by a triangle $T$ with vertices $r_i$ representing the pair $(w^i,g^i)$ for $i \in \{0,1,2\}$. For every $r \in T$, this homotopy gives a pair $(w^r, g^r)$ with $w^r = w$ and $g^r$ a metric that is standard outside $K_M \times K_E \times K_E$.

Equip $T$ with a metric $h$ that gives each edge of $T$ length one. Let $f: T \to \R$ be a nonnegative function with an index $i$ critical point at vertex $r_i$ with critical value $i\dfrac{\epsilon}{4}$ and no other critical points. Lastly suppose $\nabla_hf$ is tangent to the edges of the triangle and agrees with the standard gradient of $\epsilon \left((1/2)t^2 - (1/4)t^4\right)$ on each edge.

The remainder of the proof follows as in Lemma \ref{pathhtpy} by analyzing spaces of flow lines from $(p,r_0)$ to $(q, r_2)$ for $p \in \crit_+(w^0)$ and $q \in \crit_+(w^2)$.
\end{proof}

Lastly, since concatenating a path with its reverse is homotopic to the constant path, we need:
\begin{lem}\label{constpath}
Given a constant path $\Gamma = (w^t,g^t)$ with $w^t =w$ and $g^t = g \in \mathcal{G}_F$ for $t \in [0,1]$, $\Phi_{\Gamma} = \text{id}_{C(F)}$.
\end{lem}

\begin{proof}
The fact that every point of $\Gamma$ is a Morse-Smale pair makes this case different than just fixing the path of functions. Given $p \in \crit_+(w)$, we claim there is an isolated flow line from $(p,0)$ to $(p,1)$ along $\nabla_GW$, for $W, G$ as in \ref{nabGW}. In fact, for all $t \in [0,1]$, $$\nabla_GW(p,t) = \left( 0, \nabla \left( \epsilon \left((1/2)t^2 - (1/4)t^4 \right) \right) \right), $$ and the result follows.
\end{proof}

Thus, through the logic outlined at the beginning of this subsection, we have:
\begin{prop}
For an admissible path $\Gamma$, the map $\Phi_\Gamma$ induces an isomorphism $GH^*(F^0) \to GH^*(F^1)$. 
\end{prop}
\end{proof}


\subsection{Invariance of $GH^*(F)$ with respect to stabilization and fiber-preserving diffeomorphism}

We show that the generating family (co)homology descends to equivalence classes of generating families, as defined in Subsection \ref{ssec:gf-back}.

\begin{prop}  \label{lem:cohom-equiv}  If $F_0 \sim F_1$, then $\rgh{*}{F_0} \simeq \rgh{*}{F_1}$.
 
\end{prop}
This follows from Lemmas \ref{lem-stab} and \ref{lem-fpd}.

\begin{lem}\label{lem-stab}
If $F:M\times \R^N \rightarrow \R$ is altered by a positive or negative stabilization resulting in $\widehat{F}:M\times \R^N \times \R \rightarrow \R$ then $\rgh{*}{\widehat{F}} \cong \rgh{*}{F}$.
\end{lem}

\begin{proof}
Given a generating family $F: M \times \R^N \to \R$, define $F^{\pm}: M \times \R^N \times \R \to \R$ where $F^{\pm}(x,e,e') = F(x,e) \pm (e')^2$. It suffices to show $\rgh{*}{F^\pm} \cong \rgh{*}{F}$. 

We will denote and express the difference functions from $F^\pm$ by 
\begin{align*}
w^\pm: M \times \R^N \times \R \times \R^N \times \R &\to \R \\
(x,e_1, e_1', e_2, e_2') &\mapsto F^\pm(x,e_1,e_1') - F^\pm(x,e_2, e_2') \\
&= F(x,e_1) \pm (e_1')^2 - F(x,e_2) \mp (e_2')^2
\end{align*}

Given $p=(x,e_1,e_2) \in \crit_+(w)$ there is a corresponding critical point $p'=(x,e_1,0,e_2,0) \in \crit_+(w^\pm)$ with the same critical value and $\ind_{w^\pm}(p') = \ind_w(p) +1$. This gives a bijection between the generators of $\V{}$ and $C(F^\pm)$ and this bijection preserves grading: $|p'| = \ind_{w^\pm}(p') - (N+1) = \ind_w(p) - N = |p|$. This is precisely why the grading on $\V{}$ depends on the dimension of the fiber of $F$.

We claim that there is also a correspondence of gradient flow lines, but to show this we must choose a metric from $\mathcal{G}_{F^\pm}$. We claim that, if $g \in \mathcal{G}_F$, then $g' = g +g_0 \in \mathcal{G}_{F^\pm}$, where $g_0$ is the standard Euclidean metric on the two extra $\R$ coordinates of $M \times \R^{2(N+1)}$. The only condition from Definition \ref{wEMS} that is not immediate is the Smale condition (\ref{Smale}).

To check the Smale condition for $(w^\pm, g')$, we first show that the gradient flow of this pair splits. We may write the stabilized difference function as 
$$w^\pm(x,e_1, e_1', e_2, e_2') = w(x,e_1,e_2) \pm Q^\pm(e_1',e_2'), $$
for $Q^\pm:\R^2\to\R$ given by $Q^\pm(e_1',e_2')=  \pm (e_1')^2 \mp (e_2')^2 $, so $$d(w^\pm)_{(x,e_1, e_1', e_2, e_2')} = dw_{(x,e_1,e_2)} + dQ^\pm_{(e_1',e_2')}.$$ 
We then claim that $\nabla_{g'}w^\pm = (\nabla_gw, \nabla_{g_0}Q^\pm)$; the details of a similar proof may be found in Lemma \ref{gradsplit}. In particular, the unstable and stable manifolds split, and since $g \in \mathcal{G}_F$, this reduces to checking the Smale condition for $(Q^\pm,g_0)$. But the only critical point of $Q^\pm$ is $0=(0,0)$, and the only point in $T_0W_0^-(Q^\pm) \cap T_0W^+_0(Q^\pm)$ is $0$, and the result holds; see Prop \ref{Smaleprop} for a similar argument with more details.

To show that $\rgh{*}{F} \cong \rgh{*}{F^\pm}$, we show that, with the metric chosen, $\M(p,q) \cong \M(p',q')$, where $p',q' \in \crit_+(w^\pm)$ are the images of $p,q \in \crit_+(w)$ under the bijection described earlier in this proof. Since we showed in the previous subsection that the construction of $\rgh{*}{F^\pm}$ does not depend on the metric chosen from $\mathcal{G}_{F^\pm}$, the result will follow. In fact, since $\M(p',q')\cong \left( W^-_{p'}(w^\pm) \cap W^+_{q'}(w^\pm) \right) /\R$, the fact that $T_0W_0^-(Q^\pm) \cap T_0W^+_0(Q^\pm) = \{(0,0)\}$ gives a diffeomorphism.
\end{proof}

\begin{lem}\label{lem-fpd}
If $F:M\times \R^N \rightarrow \R$ is altered by fiber-preserving diffeomorphism that is an isometry outside $K_M \times K_E \times K_E$ resulting in $\widehat{F}:M\times \R^N \times \R \rightarrow \R$ then $\rgh{*}{\widehat{F}} \cong \rgh{*}{F}$.
\end{lem}

This result will follow from: 
\begin{lem} \label{gradvfcorres}
Suppose $g$ is a Riemannian metric on $X$, $f:X \rightarrow \R$ and $\Phi:X \rightarrow \widehat{X}$ is a diffeomorphism. If $V$ is the gradient vector field of $f$ with respect to $g$, then $\Phi_*V$ is the gradient vector field of $(\Phi^{-1})^*f$ with respect to the pullback metric $(\Phi^{-1})^*g$.
\end{lem}

\begin{proof}
Given that the vector field $V$ is such that for all $x \in X$,  $g_x (V_x,u) = df_x(u)$ for all $u \in T_x X$, we wish to show that the vector field $\Phi_*V$ satisfies $((\Phi^{-1})^*g)_{\widehat{x}}((\Phi_*V)_{\widehat{x}}, \widehat{u}) = d(f \circ \Phi^{-1})_{\widehat{x}}(\widehat{u})$ for all $\widehat{x} \in \widehat{X}$ and $\widehat{u} \in T_{\widehat{x}}\widehat{X}$.

Let $\widehat{x} \in \widehat{X}$ and $\widehat{u} \in T_{\widehat{x}}\widehat{X}$. Since $\Phi$ is a diffeomorphism, $\widehat{x}=\Phi(x)$ for some $x \in X$ and $\Phi_*$ gives an isomorphism between $T_xX$ and $T_{\widehat{x}}\widehat{X}$, so $\widehat{u} = \Phi_*u$ for some $u \in T_xX$. Thus

\begin{align*}
&((\Phi^{-1})^*g)_{\widehat{x}}((\Phi_*V)_{\widehat{x}}, \widehat{u}) \\
&=((\Phi^{-1})^*g)_{\Phi(x)}((\Phi_*V)_{\Phi(x)}, \Phi_*u)  \\
&= g_x(\Phi^{-1}_*((\Phi_* V)_x), \Phi^{-1}_* (\Phi_* u)) \\
&= g_x(V_x, u) = df_x(u) \\
&= df_x \circ d\Phi^{-1}_{\Phi(x)} (\Phi_*u) \\
&= d(f \circ \Phi^{-1})_{\Phi(x)} (\Phi_*u) \\
&= d(f \circ \Phi^{-1})_{\Phi(x)} (\widehat{u}),
\end{align*}
as desired.
\end{proof}

Lemma \ref{gradvfcorres} will give bijections of trajectories on the chain level that shows $\rgh{*}{F} = \rgh{*}{\widehat{F}}$ as long as $(\Phi^{-1})^*g \in \mathcal{G}_{\widehat{F}}$. Since $\Phi$ is an isometry outside $K_M \times K_E \times K_E$, if $g$ is Euclidean outside this set, so is $(\Phi^{-1})^*g$. Lemma \ref{gradvfcorres} induces a diffeomorphism between the stable and unstable manifolds from the flows of before and after the fiber-preserving diffeomorphism, and the Smale condition holds.

\subsection{$\rgh{*}{F}$ as a Legendrian Invariant}\label{ssec:rghLeginv}
For a given Legendrian submanifold $\leg \subset J^1M$, consider the following set:
$$\lingf(\leg) \coloneqq \{ F \mid   F \text{ is a linear-at-infinity generating family for } \leg \}.$$ As we have shown in the previous section, $GH^*(F)$ is invariant under stabilization and fiber-preserving diffeomorphism, and so our interest of invariance is over equivalence classes in $\lingf(\leg)$. In general, the set $\lingf(\leg)$ is not well understood, though see \cite{chv-pushkar, f-r,henry:mcs} for some results. As one example, when $\leg$ is a Legendrian unknot  with maximal Thurston-Bennequin invariant in $\R^3$ equipped with the standard contact structure, all elements of $\lingf(\leg)$ are equivalent; see \cite{lisa-jill}. Note that in \cite{lisa-jill}, the focus was on generating families that are linear-quadratic-at-infinity.  As in Lemma \ref{lem:lin-diff} and Lemma~\ref{lem:lq-l} linear-quadratic-at-infinity functions are equivalent to linear-at-infinity ones.

Before forming an invariant of a Legendrian submanifold $\leg$ using a generating family, it is necessary to first know that the existence of a linear-at-infinity generating family persists as $\leg$ undergoes a Legendrian isotopy.  The following proposition can be shown using
Chekanov's ``composition formula'' \cite{chv:quasi-fns}; see, for example, \cite{lisa-jill}. 
 
\begin{prop} [Persistence of Generating
  Families] \label{prop:leg-persist} For $M$ compact, let $\leg^t \subset J^1M$ be an isotopy of Legendrian submanifolds for $t \in [0,1]$.  If $\leg^0$ has a linear-at-infinity generating family $F$, then $\leg^t$ lifts to a smooth path $F^t: M \times \R^N \to \R$ where $F^t$ is a generating family for $\leg^t$, $F^0$ is obtained from $F$ by stabilization, and $F^t = F^0$ outside a compact set.
\end{prop}
  
\begin{rem} This paper also considers generating families for compact Legendrians in $J^1\R^n$.  The Persistence Proposition still holds since these Legendrians can be thought of as living in $J^1S^n$, and the linear-at-infinity condition allows generating families for such Legendrians to be defined on the domain $S^n \times \R^N$.  \end{rem}

\begin{prop} \label{cor:end-naturality} If $\leg^t \subset J^1M$ is an isotopy of Legendrian submanifolds for $t \in [0,1]$, then for the path $F^t \in \lingf(\leg)$ as in Proposition \ref{prop:leg-persist} there exists an isomorphism $\Phi^*: GH^*(F^0) \to GH^*(F^1)$.
\end{prop}

 \begin{proof}
These isomorphisms may be constructed using a continuation argument as in Proposition \ref{metricinvrgh}. Given a contact isotopy and generating family, let $F^t: M\times \R^N \to \R$ be a smooth path of generating families as in Proposition \ref{prop:leg-persist}. Given $g^0 \in \mathcal{G}_{F^0}$ and $g^1 \in \mathcal{G}_{F^1}$, construct a path of metrics $g^t$ for $t \in [0,1]$ on $M \times \R^{2N}$ so that $g^t$ is standard outside the nonlinear support compact set $K_M^t \times K_E^t \times K_E^t$. These sets vary smoothly due to the smoothness of the path $F^t$. The rest of the proof proceeds as in Proposition \ref{metricinvrgh}.
\end{proof}
%

The above proof gives an isomorphism between $\rgh{*}{F^0}$ and $\rgh{*}{F^1}$ that arise as a lifted path of generating families from a Legendrian isotopy. This isomorphism is independent for paths in the homotopy class of the given path $F^t$, but given any two generating families of isotopic Legendrians, there need not be a path between them\footnote{See \cite{ss:pi-k} for some results on homotopy spaces of generating families for Legendrian submanifolds.}.

In other words, since we might not have that all elements in $\lingf(\leg)$ are equivalent, the generating family homology of a linear-at-infinity generating family $F$ is not itself an invariant of the corresponding Legendrian $\leg$. The approach in Proposition ~\ref{cor:end-naturality} gives an alternate proof to the following:

\begin{prop} [\cite{lisa-jill, lisa:links}] For a compact Legendrian submanifold $\leg$ of $J^1M$, the set of generating family cohomology groups $$\mathcal{GH}^k(\leg) = \{ GH^k([F])\mid F \in \lingf (\leg)\},$$ is invariant under Legendrian isotopy.  \end{prop}

\section{Extended Difference Functions}
\label{sec:extdiff}

As seen in the previous section, gradient flow lines from a single difference function are used to construct generating family cohomology groups.  To form gradient flow trees from a generating family, we will need intersecting gradient trajectories, so we use Sabloff's idea of using multiple ``extended difference functions," sketched by Henry and Rutherford in \cite{hr:dga}. In this section, we define these functions and corresponding metrics which will give an identification of gradient flow lines of these spaces with those of the original difference functions $w$.

\begin{defn}\label{extdiff}  Suppose $F: M \times \R^N \to \R$ is a generating family for $\leg$.  
Let $P_3 = M \times \R^{3N}$.
For each $1 \leq i < j \leq 3$,
the \textit{extended difference functions}
$\w{i}{j}{3}: P_3 \to \R$ are defined as
$$\w{i}{j}{3}(x, e_1, e_2, e_3) = F(x, e_i) - F(x,e_j) + \begin{cases} 
      Q(e_k), & \; k < i \text{ or } k>j \\
      -Q(e_k), & \; i<k<j 
      
   \end{cases},$$
where $Q: \R^N \to \R$ is a smooth function with exactly one nondegenerate critical point $0_Q$ of index $0$ with critical value 0. 
We also require that, outside a compact set, $Q(e_k) = e_k^2$, where $e_k^2 \coloneqq \norm{e_k}^2 = e_{k1}^2 + \cdots + e_{kN}^2$. 

The set of positive-valued critical points of $\w{i}{j}{3}$ will be denoted by $\Crit_+(\w{i}{j}{3}) \subset P_3$.
\end{defn}

\begin{rem}\label{Qissues}
\begin{enumerate}
\item The number 3 in the notation of the extended difference functions $\w{i}{j}{3}$ is a bit superfluous at this stage, but it will be useful in future work to have generalizable notation.

\item We may think of the extended functions as quadratically stabilized difference functions, that is:
$$\begin{aligned}
\w{1}{2}{3}(x, e_1, e_2, e_3) &= F(x, e_1) - F(x, e_2) + e_3^2,\\
\w{2}{3}{3}(x, e_1, e_2, e_3)  &= F(x, e_2) - F(x, e_3) + e_1^2,\\
\w{1}{3}{3} (x, e_1, e_2, e_3) &= F(x, e_1) - F(x, e_3) - e_2^2.
\end{aligned}$$
   
Because of this, we might abuse notation and refer to $0_Q$ as $0$. We will see in Subsection \ref{fibpresdif} why the more general notion of ``quadratic" stabilization is necessary. In short, we wish for the extended difference functions to remain so after precomposing with a fiber-preserving diffeomorphism. We will refer to the function $Q$ as ``quadratic-like." 
\item A similar development of generating family homology uses two Legendrians $\Lambda_i$ with generating families $F_i:M \times \R^{N_i} \to \R$ for $i=1,2$ and difference function $w:M \times \R^{N_1} \times \R^{N_2} \to \R$ defined by $w(x,e_1,e_2) = F_1(x,e_1) - F_2(x,e_2)$ whose critical points correspond to Reeb chords between $\Lambda_1$ and $\Lambda_2$ (see, for example, \cite{lisa-jill}). Assuming, after possibly stabilizing the generating families, that $N_1=N_2$, one can define similar extended difference functions and continue this paper in that setting. This is analogous to the bilinearized Legendrian contact homology of \cite{bc:bilinearized}.

\end{enumerate}
\end{rem}

Even though we will be working with multiple functions, we will form the product on $\V{}$ as in Definition \ref{gradedvs}. While the sets of positive-valued critical points of the extended difference functions $\w{i}{j}{3}$ are different, there is a natural way to identify them each with positive-valued critical points of the original difference function $w$. 

\begin{lem}\label{iotabij}
For $1 \leq i < j \leq 3$, there are bijections:
$$\iota_{i,j;3}: \crit_+(w) \to \crit_+(\w{i}{j}{3})$$
which preserve critical value. In addition, we have the following index relation:
\begin{equation}\label{eqn:index}
|p| = \ind w(p) - N = \ind \w{i}{j}{3} \left( \iota_{i,j;3} (p) \right) - (j-i)N.
\end{equation}

\end{lem}
\begin{proof}
The bijections are defined as follows: 
\begin{align}\label{eqn:crit-pts}
\begin{split}
 \iota_{1,2;3}: \crit_+(w) &\to \crit_+(\w{1}{2}{3}), \\
 (x, e, e') &\mapsto (x, e, e', 0_Q)\\
 \iota_{2,3;3}: \crit_+(w) &\to \crit_+(\w{2}{3}{3}), \\
 (x, e, e') &\mapsto (x, 0_Q, e, e')\\
 \iota_{1,3;3}:\crit_+(w) &\to \crit_+(\w{1}{3}{3}),\\
 (x, e, e') &\mapsto (x, e, 0_Q, e'),
 \end{split}
 \end{align}
If $(x, e, e') \in \V{\ell}$ then $(x, e, e') \in  \Crit_+{w}$ with Morse index $\ell + N$. From the definition of the extended difference functions in Definition \ref{extdiff}, we see immediately that $\iota_{i,j;3}(x, e, e') \in \Crit_+{\w{i}{j}{3}}$. The index of $\iota_{1,2;3}(x, e, e')$ and $\iota_{2,3;3}(x, e, e')$ remains $\ell +N$, while there are $N$ extra subtracted quadratic terms in the extended difference function $\w{1}{3}{3}$ so $\iota_{1,3;3}(x, e, e')$ has index $\ell +N +N = \ell + 2N$. Since we have added or subtracted terms of functions with critical value 0, the critical values will not change.
\end{proof}

\begin{rem}\label{abusenot}
Since every critical point $p$ of $w$ of positive critical value will correspond to a Reeb chord of the Legendrian $\leg$ generated by $F$, the same is true of critical points $\iota_{i,j;3}(p)$ of $\w{i}{j}{3}$. The positive critical value of a critical point $p$ (resp. $\iota_{i,j;3}(p)$) will be the
length of the corresponding Reeb chord. By an abuse of notation, we will often use $p$ to denote both $p$ and $\iota_{i,j;3}(p)$. 
\end{rem}

\begin{defn}
By Definition \ref{defn:li}, if $F:M \times \R^N \to \R$ is a linear-at-infinity generating family, then we may write $F(x,e) = F^c(x,e)+A(e)$ where $F^c:M \times \R^N \to \R$ is compactly supported on $K_M \times K_E \subseteq M \times \R^N$. We assume that $0_Q \in K_E$ and that $K_E$ is convex; if not, enlarge the compact set so that this is true. Similarly, we assume that $Q$ is quadratic outside $K_E$ (see Definition \ref{extdiff}). We call the compact set 
\begin{equation}\label{defK}
K \coloneqq K_M \times K_E \times K_E \times K_E \subset P_3 
\end{equation}
 the \textit{non-linear support} of the extended difference functions $\w{i}{j}{3}$. 
\end{defn}

\begin{rem}\label{critscmpct} 
Suppose that $F: M \times \R^N \to \R$ is a linear-at-infinity generating family that agrees with a non-zero linear
function outside $K_M \times K_E$, for compact sets $K_M \subset M$ and $K_E \subset \R^N$.   Then 
every critical point of an 
extended difference function
$w_{i,j;3}: M \times \R^{3N} \to \R$ is of the form $(x,e_1, e_2,e_3)$ where $x \in K_M$, $e_i, e_j \in K_E$, and $e_k = 0_Q$ for $k \neq i,j$. Thus, every point in $\crit_+(\w{i}{j}{3})$ is contained in the set $K$; see Lemma \ref{wcritscmpct}.
%
\end{rem}

\begin{defn}\label{EMS}
Given  $Q: \R^N \to \R$ as in Definition \ref{extdiff}, let $\mathcal{G}_Q$ denote the set of Riemannian metrics $g_Q$ on $\R^N$ such that $g_Q$ is the standard Euclidean metric outside $K_E$ and in a neighborhood of $0_Q$, the unique critical point of $Q$. 

Given $g_w \in \mathcal{G}_F$ and $g_Q \in \mathcal{G}_Q$, we define the following three ``split" metrics $g_{i,j;3}$ pointwise on $P_3$:
\begin{align*}
(g_{1,2;3})_{(x,e_1,e_2,e_3)} &= (g_w)_{(x,e_1,e_2)} + (g_Q)_{e_3} \\
(g_{2,3;3})_{(x,e_1,e_2,e_3)} &= (g_w)_{(x,e_2,e_3)} + (g_Q)_{e_1} \\
(g_{1,3;3})_{(x,e_1,e_2,e_3)} &= (g_w)_{(x,e_1,e_3)} + (g_Q)_{e_2}.
\end{align*}
\end{defn}

The metrics  $g_{i,j;3}$ from Definition \ref{EMS} produce gradient vector fields of the extended difference functions that we may express in terms of gradient vector fields of the original difference function, $w$. To see this, first recall that the extended difference functions are of the following form (see Definition \ref{extdiff}):
 $$\w{i}{j}{3}(x, e_1, e_2, e_3) = w(x, e_i, e_j) \pm Q(e_k),$$
 where $k \in \{1,2,3\}$ is such that $k \neq i,j$. 

Fix a point $p=(x,e_1,e_2,e_3) \in P_3$. Then for any permutation of $i,j,k \in \{1,2,3\}$ with $i<j$, $$T_pP_3 = T_{(x,e_i, e_j)}(M \times \R^N \times \R^N) \times T_{e_k} \R^N$$ and we have that $$d (\w{i}{j}{3})_p = dw_{(x,e_i,e_j)} \pm dQ_{e_k}.$$

Putting this all together, we have the following Lemma:

\begin{lem}\label{gradsplit}
Given $g_w \in \mathcal{G}_F$ and $g_Q \in \mathcal{G}_Q$, let $g_{i,j;3}$ as in Definition \ref{EMS}. Then, up to a reordering of coordinates, we have the following split of gradient vector fields:
$$\nabla_{g_{i,j;3}} \w{i}{j}{3} = \left( \nabla_{g_w} w, \nabla_{g_Q} Q \right).$$
\end{lem}

\begin{proof}
Fix $g_w \in \mathcal{G}_F$ and $g_Q \in \mathcal{G}_Q$. By definition, $\nabla \w{i}{j}{3} = \nabla_{g_{i,j;3}} \w{i}{j}{3}$ is the unique vector field so that for all $p=(x,e_1,e_2,e_3) \in P_3$, $g_p(\nabla \w{i}{j}{3}, \cdot) = d (\w{i}{j}{3})_p (\cdot) = \left( dw_{(x,e_i,e_j)} + dQ_{e_k} \right) (\cdot)$. 

Let $v = (v_w,v_e) \in T_pP_3=  T_{(x,e_i, e_j)}(M \times \R^N \times \R^N) \times T_{e_k} \R^N$. We check:
\begin{align*}
&g_p\left( (\nabla_{g_w} w, \nabla_{g_Q} Q), (v_w,v_e) \right) \\
&= (g_w)_{(x, e_i, e_j)}(\nabla w, v_w) + (g_Q)_{e_k}(\nabla Q, v_e) \\
&= d w_{(x, e_i, e_j)}(v_w) + d Q_{e_k}(v_e) \\
&= d_p(\w{i}{j}{3})(v_w,v_e).
\end{align*}
which, due to the uniqueness of gradient vector fields, implies the result.
\end{proof}

The preceding Lemma shows why we chose metrics as in Definition \ref{EMS}. We must check, however, that such a choice of metric yields Morse-Smale pairs.

\begin{prop}\label{Smaleprop}
Given $g_w \in \mathcal{G}_F$ and $g_Q \in \mathcal{G}_Q$, each $(\w{i}{j}{3}, g_{i,j;3})$ satisfies the Smale condition on positive-valued critical points: for every pair of critical points $p,q \in \Crit_+{\w{i}{j}{3}}$, the unstable and stable manifolds of $p$ and $q$ have a transverse intersection.
\end{prop}


\begin{proof}
Fix $p = \iota_{i,j;3}(p')$ and $q = \iota_{i,j;3}(q')$ for $p',q' \in \V{}$, and suppose $a=(x,e_1,e_2,e_3) \in W^-_{p}(\w{i}{j}{3}) \cap W^+_{q}(\w{i}{j}{3})$. 

Lemma \ref{gradsplit} implies that the flow $\Psi$ of $\nabla_{g_{i,j;3}}\w{i}{j}{3}$ on $P_3$ may be expressed as $\Psi=(\Psi_w,\Psi_Q)$, where $\Psi_w$ is the flow of $\nabla_{g_w}w$ and $\Psi_Q$ is the flow of $\nabla_{g_Q}Q$. This implies that $W^-_{p}(\w{i}{j}{3}) = W^-_{p'}(w) \times W^-_0(Q)$ and $W^+_{q}(\w{i}{j}{3}) = W^+_{q'}(w) \times W^+_0(Q_{i,j;3})$, where we use 0 to represent $0_Q$, the only critical point of $Q$. 

Thus, we have that
\begin{align*}
&T_aW^-_{p}(\w{i}{j}{3}) + T_aW^+_{q}(\w{i}{j}{3}) \\
&=T_a(W^-_{p'}(w) \times W^-_0(Q)) + T_a( W^+_{q'}(w) \times W^+_0(Q)) \\
&=\left( T_{(x,e_i,e_j)}W^-_{p'}(w) \times T_{e_k}W^-_0(Q) \right) + \left( T_{(x,e_i,e_j)}W^+_{q'}(w) \times T_{e_k}W^+_0(Q) \right) \\
&=\left(T_{(x,e_i,e_j)}W^-_{p'}(w) + T_{(x,e_i,e_j)}W^+_{q'}(w) \right) \times \left(T_{e_k}W^-_0(Q) + T_{e_k}W^+_0(Q) \right) \\
&= T_{(x,e_i,e_j)}(M \times \R^N \times \R^N) \times \left(T_{e_k}W^-_0(Q) + T_{e_k}W^+_0(Q) \right),
\end{align*}
where the first term in the final equivalence is our assumption that $(w,g)$ satisfies the Smale condition for points in $\crit_+(w)$. For the second term, note that $e_k \in W^-_0(Q) \cap W^+_0(Q)$ implies that $e_k =0_Q$. Since  $T_0W^-_0(Q) + T_0W^+_0(Q) =T_0\R^N = \R^N$, we have that 
$$T_aW^-_{p}(\w{i}{j}{3}) + T_aW^+_{q}(\w{i}{j}{3}) =  T_{(x,e_i,e_j)}(M \times \R^N \times \R^N) \times T_{e_k}\R^N = T_aP_3,$$
as desired.
\end{proof}

We can now define infinite trajectory spaces of the extended difference functions.

\begin{defn}\label{inftraj}
For $p_-, p_+ \in \crit_+(w)$, the \dfn{unbroken infinite Morse trajectory space} between $p_-$ and $p_+$ is
\begin{align*}
\Mtofro{i}{j}{p_-}{p_+} \coloneqq  \{  &\gamma: (- \infty, \infty) \rightarrow P_3 \mid \dot{\gamma} = \nabla_{g_{i,j;3}} \w{i}{j}{3}, \\ &\lim_{t \to - \infty} \gamma(t) = \iota_{i,j,3}(p_-), \lim_{t \to \infty} \gamma(t) = \iota_{i,j,3}(p_+)\}  / \R, 
\end{align*}
where $/ \R$ denotes quotienting by the action of $\R$ that takes $\gamma(t)$ to $\gamma(t+a)$ for $a \in \R$.
\end{defn}

Given the correspondence of the (positive-valued) critical points of $w$ and the extended difference functions $\w{i}{j}{3}$, we would like there to also be a correspondence of gradient flow lines. This is where we see the benefit of choosing our metrics $g_{i,j;3}$ to be ``split" as in Definition \ref{EMS}.

\begin{prop}
For appropriate choice of metrics, there are bijections
$$\M(p,q) \leftrightarrow \M_{i,j;3}(p,q)$$
for each $1 \leq i < j \leq 3$.
\end{prop}

\begin{proof}
Express $\gamma: \R \to M \times \R^N \times \R^N \in \M(p,q)$ as $\gamma(t) = \left( a(t), b_1(t), b_2(t) \right)$ for $a:\R \to M$ and $b_1,b_2: \R \to \R^N$. Define paths $\gamma_{i,j;3}: \R \to P_3$ by
\begin{align*}
\gamma_{1,2;3}(t)&= \left(a(t), b_1(t), b_2(t), 0_Q \right) \\
\gamma_{2,3;3}(t)&= \left(a(t), 0_Q, b_1(t), b_2(t) \right) \\
\gamma_{1,3;3}(t)&= \left(a(t), b_1(t), 0_Q, b_2(t) \right) .
\end{align*}
We claim that $\gamma_{i,j;3} \in \M_{i,j;3}(p,q)$ and that this identification defines a bijection (that is, up to reparametrization, all trajectories in $\M_{i,j;3}(p,q)$ are of this form). 

Lemma \ref{gradsplit} implies that gradient trajectories of $\w{i}{j}{3}$ may be written in terms of a gradient trajectory of $w$ and one of $Q:\R^N \to \R$. Since $0_Q$ is the only critical point of $Q$, the constant trajectory at $0_Q$ is its only gradient trajectory.
\end{proof}

\section{Moduli Space of Gradient Flow Trees}\label{sec:moduli}

We will study positive gradient flow lines of the extended difference functions $\w{i}{j}{3}$ with respect to metrics as in Definition \ref{EMS}.  Gradient flow lines are well-studied objects in Morse Theory, and we will work with moduli spaces of intersecting flow lines, which we will refer to as \textit{gradient flow trees}. Understanding the structure of these spaces will play an integral role in defining our product.

In particular, to define products with correct properties, we will need to show that our moduli spaces are smooth manifolds with certain compactification properties. We consider gradient flow trees consisting of three half-infinite gradient trajectories, one from each extended difference function, that limit to critical points at their infinite ends and intersect at their finite ends. To achieve transversality of this intersection, we consider trees that ``almost" intersect at their finite ends, up to a small fixed vector at each finite end.

 \begin{figure}[htb]
\includegraphics[width=1 \textwidth]{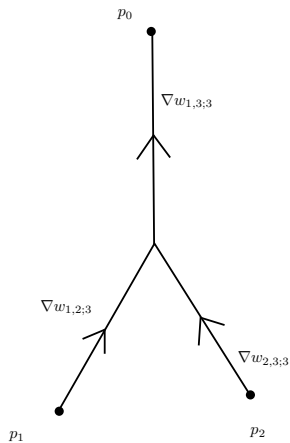} 
\caption{A gradient flow tree with three intersecting half-infinite trajectories.}
\label{ytree}
\end{figure} 

To define the space of flow trees, we use results from Wehrheim \cite{wehrheim:morse} in which spaces of broken, half-infinite gradient trajectories of a Morse-Smale pair $(f,g)$ are equipped with a metric space structure and shown to be smooth manifolds with corners. In Wehrheim's setup, $f$ is a Morse function on closed manifold. Our setup differs in that our functions are Morse-Bott rather than Morse and defined on a noncompact space. However, the critical points we are interested in, the positive-valued ones, are isolated and contained in a compact set. This will allow us to form a compact space of ``broken flow trees," i.e., flow trees made up of broken trajectories, and we will show that space has the structure of a smooth manifold with corners. Its 0-stratum will be the desired space of unbroken flow trees.


Our strategy will be to first define an ambient space $\overline{X}$ consisting of triples of broken and unbroken trajectories. We will quote results from \cite{wehrheim:morse} to show that $\overline{X}$ is metric space and has the structure of a smooth manifold with corners. Because our functions are not defined on closed manifolds, we cannot say that $\overline{X}$ is compact. We can, however, view the space of flow trees as a closed subset of an open neighborhood in $\overline{X}$ whose closure is compact. We then perturb in this neighborhood, retaining compactness, to ensure that the space of flow trees is a compact smooth manifold with corners.

To proceed with this strategy, we build a few choices into our construction, explained in the following remark:

\begin{rem}\label{shrinkfiber}
Lemma \ref{critscmpct} implies that there are only a finite number of critical points with positive critical value since such points are isolated. Thus, we know that there exists a smallest positive critical value 
\begin{equation}
\rho \coloneqq \rho_F = \min \{w(p) \; : \; p \in \V{}, w(p)>0 \}. 
\end{equation}
 To prove certain results in Section \ref{sec:moduli}(see Lemma \ref{awayfrom0}), we need to use this fact to build a couple of choices into our construction:
\begin{enumerate}
\item We shrink the fiber coordinates in the following way: we apply a fiber-preserving diffeomorphism to $P_3$ that is the identity outside of $K$ and so that every point $y = (x, e_1, e_2, e_3) \in K \subset P_3$ is such that $Q(e_1) + Q(e_2) -Q(e_3) < \rho$. In Section \ref{secinvar} we will see that the product is invariant under fiber preserving diffeomorphism, so this choice will not affect the outcome.
\item Since the set $K$ from Definition \ref{defK} is compact, each $\w{i}{j}{3}|_K$ is uniformly continuous. In particular, for $\rho$ as above, there exists $\delta = \min\{\delta_{1,2;3}, \delta_{2,3;3}, \delta_{1,3;3}\} > 0$ such that for all $y_1, y_2 \in K$, $|y_1 - y_2| < \delta_{i,j;3} $  implies that $|\w{i}{j}{3}(y_1) - \w{i}{j}{3}(y_2)| < \frac{\rho}{4}.$  
\end{enumerate}
\end{rem}

\begin{defn}\label{halfinftraj}
%

The \textit{unbroken half-infinite Morse trajectory spaces} to/from a critical point $p \in \crit_+(w)$ are defined as:
$$\Mto{i}{j}{3}{p} \coloneqq \{ \gamma: [0, \infty) \rightarrow P_3 \mid \dot{\gamma} = \nabla \w{i}{j}{3}, \lim_{t \to \infty} \gamma(t) = \iota_{i,j,3}(p) \} \text{ and} $$
$$\Mfro{i}{j}{3}{p} \coloneqq \{ \gamma: (-\infty, 0] \rightarrow P_3 \mid \dot{\gamma} = \nabla \w{i}{j}{3}, \lim_{t \to -\infty} \gamma(t) = \iota_{i,i,;3}(p) \}. $$
\end{defn}

\begin{rem}\label{smothstr}
\begin{enumerate}
\item The sets in Definitions \ref{halfinftraj} and \ref{inftraj} inherit smooth structures from unstable and stable manifolds:
\begin{align*}
\Mtofro{i}{j}{p_-}{p_+} &\cong \left( W_{p_-}^-(\w{i}{j}{3}) \cap W_{p_+}^+(\w{i}{j}{3}) \right) / \R,\\
\Mto{i}{j}{3}{p} &\cong W_p^+(\w{i}{j}{3}),\\
\Mfro{i}{j}{3}{p} &\cong W_p^-(\w{i}{j}{3}).
\end{align*} 
\item Quotienting by reparametrization is not needed for half-infinite trajectories because the finite endpoint at 0 and hence the image of the trajectory changes under reparametrization. 

\item We may restrict where the finite end of the half-infinite trajectories lies. For $U \subset P_3$, we use $\M_{i,j;3}(U,p)$ to be the subset of $\Mto{i}{j}{3}{p}$ where $\gamma(0) \in U$. Then $\M_{i,j;3}(U,p) \cong W_p^+(\w{i}{j}{3}) \cap U$. A similar statement holds for $\M_{i,j;3}(p,U)$.
\end{enumerate}

\end{rem}

We will also consider broken half-infinite Morse trajectory spaces. These are sequences of trajectories, one of which is half-infinite and the rest infinite. To ease notation, let $\{ \mathcal{U}_-, \mathcal{U}_+ \} = \{ p, P_3 \}$ or $\{ p, U \}$. That is, we will consider sequences where one endpoint is a critical point and the other is a point in $P_3$ or $U$. 

\begin{defn}\label{defstrata}
 We define the \textit{$\ell$-fold broken half-infinite trajectories} to be
$$\overline{\M}_{i,j;3}(\mathcal{U}_-, \mathcal{U}_+)_\ell \coloneqq \bigcup \M_{i,j;3}(\mathcal{U}_-,p_1) \times \M_{i,j;3}(p_1,p_2) \times \dots \times \M_{i,j;3}(p_{\ell}, \mathcal{U}_+),$$
where the union is taken over sequence of critical points $p_1, \dots, p_{\ell} \in \crit_+(w)$ such that $\M_{i,j;3}(\mathcal{U}_-,p_1), \M_{i,j;3}(p_1,p_2), \dots ,\M_{i,j;3}(p_{\ell}, \mathcal{U}_+) \neq \emptyset$.
\end{defn}

\begin{defn}\label{gentraj}
The \textit{generalized Morse trajectory space} is
$$\overline{\M}_{i,j;3}(\mathcal{U}_-, \mathcal{U}_+) \coloneqq \bigcup_{\ell \in \mathbb{N}} \overline{\M}_{i,j;3}(\mathcal{U}_-, \mathcal{U}_+)_\ell.$$
We will use $\overline{\gamma} = \{\gamma_1, \dots , \gamma_\ell\}$ to denote an element of $\overline{\M}_{i,j;3}(\mathcal{U}_-, \mathcal{U}_+)$.
\end{defn}

\begin{rem}
The union in Definition \ref{gentraj} is finite: the space $\V{}$ is generated by a finite set of critical points and all critical points live in a compact subset of $P_3$ (see Lemma \ref{critscmpct}). The finite ends of the generalized trajectories may leave the non-linear support set $K$, so these spaces are not necessarily contained in a compact set.
\end{rem}

To form the space $\overline{X}$ we will use specific choices of the pair $(\mathcal{U}_-, \mathcal{U}_+)$ for each pair $(i,j)$ with $1 \leq i < j \leq 3$. Recall that we are interested in forming gradient flow trees with two branches that follow the flow from two positive-valued critical points along positive gradient vector fields of $\w{1}{2}{3}$and $\w{2}{3}{3}$ and a branch following the flow of a positive gradient vector field of $\w{1}{3}{3}$ to a positive-valued critical point. (See Figure \ref{ytree}). As we are only interested in broken trajectories that break at positive-valued critical points, we must restrict the finite end of the trajectory flowing to a critical point along $\nabla \w{1}{3}{3}$ to be positive. For our purposes, we use
 $$( \mathcal{U}_-, \mathcal{U}_+ ) = ( p_1, P_3 ), ( p_2, P_3 ), \text{ or } \left(\{ \w{1}{3}{3} > \frac{\rho}{8} \}, p_0\right),$$
 where $\rho$ is the least positive critical value of $w$, as defined in Remark \ref{shrinkfiber}. \\
 
  In fact, with the fiber-preserving diffeomorphism we have performed as described in Remark \ref{shrinkfiber}, this is not a restricting choice. The following Lemma shows that the trajectories that will make up gradient flow trees are contained in the specified half-infinite trajectory spaces.
 
\begin{lem}\label{awayfrom0pre}
Let $p_1, p_2 \in \V{}$. If $(\gamma_1, \gamma_2, \gamma_3) \in \M_{1,2;3}(p_1, P_3) \times \M_{2,3;3}(p_2, P_3) \times \M_{1,3;3}(P_3, p_0)$ satisfies $\gamma_1(0) = \gamma_2(0) =\gamma_3(0)$, then $\gamma_3(0) > \rho >0$, where $\rho$ is the least positive critical value of $w$. In particular, this shows that $p_0 \in \V{}$ as well.
\end{lem}

\begin{proof}
Let $$y=(x^y,e_1^y,e_2^y,e_3^y)= \gamma_1(0) = \gamma_2(0) =\gamma_3(0). $$

Since each $\gamma_i$ follows a positive gradient flow of an extended difference function, we have that $\w{1}{2}{3}(y) \geq \w{1}{2}{3}(p_1)$, $\w{2}{3}{3}(y) \geq \w{2}{3}{3}(p_2),$ and $\w{1}{3}{3}(y) \leq \w{1}{3}{3}(p_0).$

By construction of the extended difference functions,
$$\w{1}{3}{3}(y) = \w{1}{2}{3}(y) + \w{2}{3}{3}(y) - \left(Q(e_3^y) + Q(e_1^y) - Q(e_2^y)\right).$$

We now make use of a choice we built into our constructions of $\modsp$; see Remark \ref{shrinkfiber}. 
\begin{align*}
 &\w{1}{3}{3}(y) \\
&=  \w{1}{2}{3}(y) + \w{2}{3}{3}(y) - \left(Q(e_3^y) + Q(e_1^y) - Q(e_2^y)\right) \\
&> \w{1}{2}{3}(p_1)+ \w{2}{3}{3}(p_2) - \rho \\
&\geq 2\rho - \rho  > 0. \qedhere
\end{align*} 

\end{proof} 

\begin{rem}
\begin{enumerate}
\item It sufficed to state the preceding lemma in terms of unbroken trajectories, but the result holds for broken trajectories as well, since an unbroken tree exists around the intersection point of three broken trajectories. 
\item While a larger bound, like $\dfrac{\rho}{2}$, would suffice at this stage, we build our trajectory spaces using $\overline{\M}_{1,3;3}\left(\left\lbrace\w{1}{3}{3} > \dfrac{\rho}{8}\right\rbrace, p_0\right)$ to have space to perturb for transversality.
\end{enumerate}
\end{rem}

There is a natural metric on $\overline{\M}_{i,j;3}(\mathcal{U}_-, \mathcal{U}_+)$:

\begin{defn}\label{dfn-metric}
On $\overline{\M}_{i,j;3}(\mathcal{U}_-, \mathcal{U}_+)$ consider the metric $d_{\overline{\M}}$ which is the Hausdorff distance on the images of broken trajectories. 

$$ d_{\overline{\M}}(\underline{\gamma}, \underline{\gamma'}) \coloneqq d_{\text{Haus}}(\overline{\text{im}\underline{\gamma}}, \overline{\text{im}\underline{\gamma'}} ) . $$

By $\overline{\text{im}\underline{\gamma}}$, we mean the closure of the union of the images of trajectories that make up the trajectory sequence $\underline{\gamma}$. Recall that the Hausdorff distance $d_{\text{Haus}}$ is a metric on non-empty compact subsets of a space defined by
$$d_{\text{Haus}}(A,B) = \text{max} \left\{ \adjustlimits\sup_{a\in A} \inf_{b\in B} \text{d}(a,b), \adjustlimits\sup_{b\in B} \inf_{a\in A}\text{d}(a,b)\right\}.$$


\end{defn}

\begin{thm}\cite[Theorem 2.3]{wehrheim:morse}  \label{KW} For $p_0, p_1, p_2 \in \crit_+(w)$, the half-infinite generalized trajectory spaces  $\left( \overline{\M}_{1,2;3}(p_1, P_3) , d_{\overline{\M}} \right)$, $\left( \overline{\M}_{2,3;3}(p_2, P_3) , d_{\overline{\M}} \right)$, and $\left( \overline{\M}_{1,3;3}(\{\w{1}{3}{3} > \dfrac{\rho}{8}\}, p_0) , d_{\overline{\M}} \right)$ are locally compact separable metric spaces that can be equipped with the structure of a smooth manifold with corners.  In each case, the $\ell$-stratum is $\overline{\M}_{i,j;3}(\mathcal{U}_-, \mathcal{U}_+)_\ell$ for the same choices $(\mathcal{U}_-, \mathcal{U}_+)$.
\end{thm}

\begin{proof}
While our setup differs from that in \cite[Theorem 2.3]{wehrheim:morse}, we argue that constructions in \cite{wehrheim:morse} suffice to claim this result. In particular, these constructions give a maximal atlas of charts and associative gluing maps to define a manifold with corners structure for Morse-Smale pairs on a closed manifold. In neighborhoods not containing critical points, there is a natural smooth structure induced by the smoothness of the gradient flow. The careful work to define the corner structure occurs in neighborhood of the critical points. Thus, while the extended difference functions are Morse-Bott and defined on a noncompact manifold, Lemmas \ref{lem:cpct-set} and \ref{awayfrom0} show that neighborhoods of the trajectories that occur in trees occur in an open set contained in compact set. Hence, the charts in \cite{wehrheim:morse} suffice to give a manifold with corners structure on the relevant trajectory spaces. In contrast to \cite[Theorem 2.3]{wehrheim:morse}, we do not have compactness of the full trajectory spaces themselves. This is exactly because the spaces $P_3$ and $\{\w{1}{3}{3} > \dfrac{\rho}{8}\}$ are not closed. We do, however, retain local compactness. More details on compactness results for these spaces are given in the proof of Theorem \ref{cornertrees}. 
\end{proof}

\begin{defn}
For $p_1, p_2, p_0 \in \V{}$, let 
$$\overline{X} \coloneqq \overline{\M}_{1,2;3}(p_1,P_3) \times \overline{\M}_{2,3;3}(p_2,P_3) \times \overline{\M}_{1,3;3}\left(\left\lbrace\w{1}{3}{3} > \dfrac{\rho}{8}\right\rbrace, p_0\right).$$
\end{defn}

Applying Lemma \ref{cornerproducts} twice shows that the space
$\overline{X}$ is a manifold with corners whose $\ell$-stratum is $$\overline{X}_\ell = \bigsqcup_{i+ j+k=\ell} \overline{\M}_{1,2;3}(p_1,P_3)_i \times \overline{\M}_{2,3;3}(p_2,P_3)_j \times \overline{\M}_{1,3;3}\left(\left\lbrace\w{1}{3}{3} > \dfrac{\rho}{8}\right\rbrace, p_0\right)_k. $$ 
In particular, its 0-stratum is \begin{equation*}\label{mfldX}
X \coloneqq \overline{X}_0 =\Mfro{1}{2}{3}{p_1} \times \Mfro{2}{3}{3}{p_2} \times \M_{1,3;3}\left(\left\lbrace\w{1}{3}{3} > \dfrac{\rho}{8}\right\rbrace, p_0\right),
\end{equation*}
and consists of triples of unbroken half-infinite Morse trajectories.

\begin{defn}\label{dfn-ev}
To record the finite endpoint of broken and unbroken half-infinite Morse trajectories, we define \textit{generalized evaluation maps}  $$\ev{1}{3}{3}{-}: \overline{\M}_{1,3;3}\left(\left\lbrace\w{1}{3}{3} > \dfrac{\rho}{8}\right\rbrace, p_0\right) \to P_3, \;\;\; \ev{i}{j}{3}{+}: \overline{\M}_{i,j;3}(p_i, P_3) \to P_3 $$
 by $$\ev{1}{3}{3}{-}(\overline{\gamma}) = \ev{1}{3}{3}{-}(\{\gamma_1, \dots ,\gamma_\ell\}) \coloneqq \gamma_1(0)$$ and $$\ev{i}{j}{3}{+}(\overline{\gamma}) = \ev{i}{j}{3}{+}(\{\gamma_1, \dots ,\gamma_\ell\}) \coloneqq \gamma_\ell(0).$$

The \textit{triple generalized evaluation map} $$\operatorname{Ev}: \overline{X} \to P_3 \times P_3 \times P_3, $$ then, is the product of these maps:
$$\text{Ev}\left(\overline{\gamma}_1,\overline{\gamma}_2,\overline{\gamma}_3\right) \coloneqq \left(\ev{1}{2}{3}{+}(\overline{\gamma}_1), \ev{2}{3}{3}{+}(\overline{\gamma}_2), \ev{1}{3}{3}{-}(\overline{\gamma}_3) \right)  $$

\end{defn}

\begin{rem}\label{evcont}
Wehrheim \cite[Lemma 3.3]{wehrheim:morse} proved that the extended evaluation maps in Definition \ref{dfn-ev} are continuous with respect to the Hausdorff metric defined in \ref{dfn-metric}. As shown in \cite[Remark 5.5]{wehrheim:morse}, the evaluation maps are smooth. Thus, $\text{Ev}$ is continuous and smooth on $\overline{X}$.
\end{rem}


\begin{defn}
We denote the \dfn{diagonal of $(P_3)^3$} by $\Delta$, i.e, 
$\Delta = \{(y,y,y) \mid y \in P_3\}.$ Then, given $p_1,p_2,p_0 \in \V{}$, a \textit{generalized flow tree} is a triple $\left(\overline{\gamma}_1,\overline{\gamma}_2,\overline{\gamma}_3\right) \in \text{Ev}^{-1}(\Delta),$ and the \textit{moduli space of generalized flow trees} is $\overline{\M}(p_1,p_2;p_0) = \text{Ev}^{-1}(\Delta)$. The \textit{moduli space of unbroken flow trees} is denoted $\M(p_1,p_2;p_0)$ and equals $\left(\text{Ev}|_X\right)^{-1}(\Delta),$ the preimage of the diagonal under the restriction of $E$ to the 0-stratum of $\overline{X}$.
\end{defn}


Although the gradient trajectories are in the non-compact space $P_3$, the following shows that all trees will have their images in a compact subset of $P_3$.

\begin{lem} \label{lem:cpct-set}
Given a linear-at-infinity generating family $F:M\times \R^N \rightarrow \R$ with non-linear support $K \subseteq P_{3}$ as in Definition \ref{defK}, for all $p_1, p_2, p_0$ and all
$\Gamma = (\overline{\gamma}_1, \overline{\gamma}_2, \overline{\gamma}_3) \in \overline{\M}(p_1,p_2; p_0)$,  $\operatorname{Im}(\Gamma) \subseteq K$.
\end{lem}

\begin{proof}
As in Remark \ref{critscmpct}, all critical points of the extended difference functions are in $K$. To simplify the proof, then, we work with unbroken trajectories. To show that the image of every $\Gamma \in \M(p_1,p_2;p_0)$ is contained in $K$, we first show that that every trajectory of  $\nabla \w{i}{j}{3}$ that leaves $K$ cannot reenter $K$.  This implies that 
 every edge $E$ in a tree, $\operatorname{image} \Gamma|_{E} $ can intersect $\partial K$ at most once.  To show that these edges in fact never intersect
 $\partial K$, we show that for the vertex  $v$ in the interior of the tree, $\Gamma(v) \in K$.  
 
 For these arguments, it will useful to first analyze some properties of $\nabla \w{i}{j}{3}$ outside a compact set.
Since $F$ is linear-at-infinity and our metrics are chosen to be standard outside $K$, we know that for $(x,e_1,e_2, e_3) \notin K$,
\begin{align*}
 \nabla \w{i}{j}{3}(x,e_1,e_2,e_3) = &\left(\dfrac{\partial F}{\partial x}(x,e_i) - \dfrac{\partial F}{\partial x}(x,e_j)\right) \dfrac{\partial}{\partial x}  \\
 &+ 
 \left(\dfrac{\partial F}{\partial e_i} (x,e_i) \right) \dfrac{\partial}{\partial e_i} 
 - \left( \dfrac{\partial F}{\partial e_j} (x,e_j) \right) \dfrac{\partial}{\partial e_j}  \\
 & \pm 2e_\ell \dfrac{\partial}{\partial e_\ell}.  \end{align*}
 where $\ell \neq i,j$ and the $\dfrac{\partial}{\partial e_\ell}$ sign is $-$ if $\ell = 2$ and $+$ else.
 
 More specifically, suppose that outside $K_M \times K_E$, $F(x, e)$ is the nonzero linear function $A(e)$ and 
$\dfrac{\partial}{\partial e} A(e) = c \in \R^N-\{0\}$.  Then we have that $\forall i,j$, 
\begin{enumerate}
\item \label{x-dir} if $x \notin K_M$, the $\dfrac{\partial}{\partial x}$ component of $\nabla \w{i}{j}{3}(x,e_1, e_2, e_3)$ equals $0$; 
\item \label{ei-dir}  if $e_i \notin K_E$, the $\dfrac{\partial}{\partial e_i}$ component of $\nabla \w{i}{j}{3}(x,e_1, e_2, e_3)$ equals $c$ ; 
\item \label{ej-dir}  if $e_j \notin K_E$, the $\dfrac{\partial}{\partial e_j}$ component of $\nabla \w{i}{j}{3}(x,e_1, e_2, e_3)$ equals $-c$;
\item \label{other-dir}  for $e_\ell$, $\ell \neq i, j$, the $\dfrac{\partial}{\partial e_\ell}$ component of $\nabla \w{i}{j}{3}(x,e_1, e_2, e_3)$ is $ 2e_{\ell}$,
when $\ell = 1$ or $\ell = 3$ and is $-2e_{\ell}$ when $\ell =2$. 
\end{enumerate}
First, suppose $\gamma$ is a trajectory of $ \nabla \w{i}{j}{3}$ and there exists a $t_0 < t_1 $ so that $\gamma(t_0) \in K, \gamma(t_1) \notin K$.  The following
argument then shows that for all $t > t_1$, $\gamma(t) \notin K$.  Since $\gamma(t_1) \notin K$, $\gamma(t_1) = (x, e_1,e_2,e_3)$ where
$x \notin K_M$ or $e_i \notin K_E$, for some $i$.    
From the form of $\nabla \w{i}{j}{k}$ outside $K$, it is easy to see that for all $t > t_1$, $\gamma(t_1) \notin K$.  For example, if $\gamma(t_1) = (x, e_1, e_2, e_3)$, where
$e_i \notin K_E$, then since  the $\dfrac{\partial}{\partial e_i}$ component of $\nabla \w{i}{j}{k}(x,e_1, \dots, e_{k+1})$ is constant or linear (and $0 \in K$), it follows
that for all $t > t_1$, the $i^{th}$ component of $\gamma(t)$ will not lie in $K_E$. 

Let $y  \in M \times \R^{3N}$ denote the intersection point of gradient trajectories. Suppose $y=(x^y, e_1^y, e_2^y , e_3^y) \notin K$. From (\ref{x-dir}),  we see that  $y \notin K$ can only follow from $e_i^y \notin K_{E}$ for some $1 \leq i \leq 3$. We complete the argument  by finding contradictions to $e_i^y \notin K_{E}$ by cases depending on $i$.

Suppose $e_1^y \notin K_{E}$. Then by (\ref{ei-dir}), we see that the $\dfrac{\partial}{\partial e_1}$ components of $\nabla \w{1}{2}{3}(x,e_1, e_2, e_3)$ and $\nabla \w{1}{3}{3}(x,e_1, e_2, e_3)$ both equal $c$, but the first flows from $K$ to $y$ and the other flows from $y$ to $K$, giving a contradiction. 

A similar contradiction is reached if $e_2^y \notin K_{E}$. In particular, the trajectories along $\nabla \w{1}{2}{3}$ and $\nabla \w{2}{3}{3}$ both flow to $y$, but by (\ref{ei-dir}) and (\ref{ej-dir}), we see the $\dfrac{\partial}{\partial e_2}$ components of the trajectories outside $K$ are constant with opposite signs.

Lastly, if $e_3^y \notin K_{E}$, we obtain a similar contradiction as in the case $i=1$ using (\ref{ej-dir}) and the fact that $\nabla \w{2}{3}{3}$ flows from $K$ to $y$ while $\nabla \w{1}{3}{3}$ flows away from $y$ back to $K$. Thus we must have $y \in K$.


Hence $\Gamma$ is contained in $K$, as desired. 
\end{proof}

Theorem \ref{KW}, Remark \ref{evcont}, and Lemma \ref{lem:cpct-set} show that $\overline{\M}(p_1,p_2;p_0)$ is compact: as $\Delta$ is closed in $P_3 \times P_3 \times P_3$, $\overline{\M}(p_1,p_2;p_0)$ is a closed subset of a locally compact space, $\overline{X}$. This, with the fact that $\overline{\M}(p_1,p_2;p_0)$ is contained in a compact subset of $P_3$, gives compactness. It remains to show that we may perturb $\overline{\M}(p_1,p_2;p_0)$ slightly in this neighborhood to obtain a compact smooth manifold with corners.

To do so, we define a perturbation ball $S$ and maps $E_{i,j;3}$ that perturb the evaluation at endpoints map by vectors in $S$. The perturbation ball $S$ and the subsequent maps are defined in slightly different ways depending on if the manifold $M$ is a Euclidean space or a compact manifold.

\begin{defn}\label{pert}
If $M=\R^n$, define the \dfn{perturbation ball} $S \subset P_3$ to be an open $\epsilon$-ball centered at 0 in $P_3 = \R^{n+3N}$. We will denote such a ball as $B^{n+3N}(\epsilon)$ or just $B(\epsilon)$ if the dimension is clear.

If $M$ is a compact manifold, then let $M \subset \R^m$ for some $m \in \mathbb{N}$. Every $B^m(\epsilon)$ defines a space $M^{\epsilon} \subset \R^m$, the open set of points in $\R^m$ of distance less than $\epsilon$ to $M$. By the $\epsilon$-Neighborhood Theorem (see, for example \cite{g-p}), if $\epsilon$ is small enough, there is a well defined submersion $\pi_M: M^{\epsilon} \to M$ that takes a point in $M^{\epsilon}$ to the unique closest point in M and is the identity when restricted to $M$. We can extended this map to get a submersion $\pi: M^{\epsilon} \times \R^{3N} \to P_3$ defined by $$\pi(x,e_1,e_2,e_3) = (\pi_M(x),e_1,e_2,e_3). $$

For $M \subset \R^m$ compact, the \dfn{perturbation ball} is $$S \coloneqq B^{m+3N}(\epsilon) \subset B^m(\epsilon) \times B^{3N}(\epsilon) \subset \R^m \times \R^{3N}.$$
\end{defn}

\begin{rem}\label{pertdelta}
For $\delta$ as in Remark \ref{shrinkfiber}, we choose the size of the perturbation ball $S$ so that for all $s \in S$, $|s| < \delta$. This choice will be used to show that an analogue of Lemma \ref{awayfrom0} still holds after perturbing $\overline{\M}(p_1,p_2;p_0)$.
\end{rem}

\begin{defn}\label{Eijbroken}
For $M=\R^n$, define the \dfn{generalized perturbed evaluation maps $E_{i,j;3}$} as follows:
\begin{align*}
E_{1,2;3}: \Mfrob{1}{2}{3}{p_1} \times S &\to P_3 \\
(\overline{\gamma}, s) &\mapsto \ev{1}{2}{3}{+}(\overline{\gamma}) + s = \gamma_\ell(0) +s, \\
E_{2,3;3}: \Mfrob{2}{3}{3}{p_2} \times S &\to P_3 \\
(\overline{\gamma}, s) &\mapsto \ev{2}{3}{3}{+}(\overline{\gamma}) + s = \gamma_\ell(0) +s, \\
E_{1,3;3}: \Mtob{1}{3}{3}{p_0} \times S &\to P_3 \\
(\overline{\gamma}, s) &\mapsto \ev{1}{3}{3}{-}(\overline{\gamma}) + s = \gamma_1(0) +s.\\
\end{align*}

If $M$ is a closed manifold, we define the \dfn{generalized perturbed evaluation maps $E_{i,j;3}$} using the map $\pi: M^{\epsilon} \times \R^{3N} \to P_3$ defined in Definition \ref{pert}:
\begin{align*}
E_{1,2;3}: \Mfrob{1}{2}{3}{p_1} \times S &\to P_3 \\
(\overline{\gamma}, s) &\mapsto \pi \left( \ev{1}{2}{3}{+}(\overline{\gamma}) + s \right)= \pi \left(\gamma_\ell(0) +s \right), \\
E_{2,3;3}: \Mfrob{2}{3}{3}{p_2} \times S &\to P_3 \\
(\overline{\gamma}, s) &\mapsto \pi \left( \ev{2}{3}{3}{+}(\overline{\gamma}) + s \right) = \pi \left( \gamma_\ell(0) +s \right), \\
E_{1,3;3}: \Mtob{1}{3}{3}{p_0} \times S &\to P_3 \\
(\overline{\gamma}, s) &\mapsto \pi \left( \ev{1}{3}{3}{-}(\overline{\gamma}) + s \right) = \pi \left( \gamma_1(0) +s\right). \\
\end{align*}

In either case, define the following \dfn{triple perturbed evaluation map}:
\begin{align*}
E: \overline{X} \times (S \times S \times S) &\to P_3 \times P_3 \times P_3 \\
\left( (\gamma_1, \gamma_2, \gamma_3) , (s_1,s_2,s_3) \right) &\mapsto \left( E_{1,2;3}(\gamma_1,s_1), E_{2,3;3}(\gamma_2,s_2), E_{1,3;3}(\gamma_3,s_3) \right).
\end{align*}
\end{defn}

%
%

\begin{rem}
The $E_{i,j;3}$ maps are well-defined: this is clear when $M= \R^n$, and for compact $M \subset \R^m$, the evaluation maps have outputs in $P_3 =M \times \R^{3N} \subset M^{\epsilon} \times \R^{3N} \subset \R^{m+3N}$. Adding an element $s \in S$ to this output will give a point within distance $\epsilon$ of the endpoint. This is a valid input for the map $\pi$, which we use to get a corresponding point $P_3$. 
\end{rem}

%
For a fixed $s = (s_1,s_2,s_3) \in S^3$, we denote $E_s^{-1}(\Delta^3)$ by $\modsp$. We will show in Theorem \ref{cornertrees} that for almost every $s \in S^3$, $\modsp$ is a smooth manifold with corners. First, we must argue that analogues of Lemma \ref{awayfrom0pre} and Lemma \ref{lem:cpct-set} still hold after perturbation by $s$. Clearly, the image of a perturbed tree will still be a subset of a compact set of $P_3$, so it only remains to show that the choice to form trees with the edge following $\nabla \w{1}{3}{3}$ living in $\overline{\M}_{1,3;3}\left(\left\lbrace\w{1}{3}{3} > \dfrac{\rho}{8}\right\rbrace, p_0\right)$ was not restrictive.

\begin{lem}\label{awayfrom0}
If $p_1, p_2 \in \V{}$ and $$\Gamma = (\gamma_1, \gamma_2, \gamma_3) \in  \M_{1,2;3}(p_1, P_3) \times \M_{2,3;3}(p_2, P_3) \times \M_{1,3;3}(P_3, p_0)$$ 
with $E_{1,2;3}(\gamma_1,s_1)= E_{2,3;3}(\gamma_2,s_2) = E_{1,3;3}(\gamma_3,s_3),$ then 
$ \gamma_3(0) > \dfrac{ \rho}{4} >0, $
where $\rho$ is the least positive critical value of $w$. 
\end{lem}

\begin{proof}
Consider
\begin{align*}
y =  E_{1,2;3}(\gamma_1,s_1)= E_{2,3;3}(\gamma_2,s_2) = E_{1,3;3}(\gamma_3,s_3) \\ =(x(y),e_1(y),e_2(y),e_3(y)) \in P_3.
\end{align*}
 While the $E_{i,j;3}$ maps were defined in slightly different ways dependent on if the underlying manifold $M$ was Euclidean or closed (see Definition \ref{Eijbroken}), we may express them as $\pi(\gamma_k +s_k)$ for $k=1,2,3$, where $\pi$ is the identity or the submersion described in the definition. 

Since the trees in $\modsp$ are defined using the positive gradient flow of the extended difference functions, we have that $\w{1}{2}{3}(\gamma_1(0)) \geq \w{1}{2}{3}(p_1)$ and $\w{2}{3}{3}(\gamma_2(0)) \geq \w{2}{3}{3}(p_2)$.

By construction of the extended difference functions,
\begin{equation}\label{jumpeq}
\w{1}{3}{3}(y) = \w{1}{2}{3}(y) + \w{2}{3}{3}(y) - \left(Q(e_3^y) + Q(e_1^y) - Q(e_2^y)\right).
\end{equation}

We now make use of a couple of the choices we have built into our constructions of $\modsp$; see Remark \ref{shrinkfiber}. Since the perturbation terms $s_i \in S$, we have ensured that $|s_i| < \delta $ so that, using the uniform continuity of $\w{i}{j}{3}$ on $K$, we have that
\begin{equation} 
 |\w{i}{j}{3} (y) - \w{i}{j}{3} \left( \gamma_k(0)\right)| < \dfrac{\rho}{4}. 
 \end{equation}
From this and ($\ref{jumpeq}$) we see that

\begin{align*}
  &\w{1}{3}{3}(\gamma_3(0)) \\
&> \w{1}{3}{3}(y) - \dfrac{\rho}{4}\\
&= \; \w{1}{2}{3}(y) + \w{2}{3}{3}(y) - \left(Q(e_3^y) + Q(e_1^y) - Q(e_2^y)\right) - \dfrac{\rho}{4} \\
&> \w{1}{2}{3}(\gamma_1(0)) - \dfrac{\rho}{4} +  \w{2}{3}{3}(\gamma_2(0)) - \dfrac{\rho}{4} - \left(Q(e_3^y) + Q(e_1^y) - Q(e_2^y)\right) - \dfrac{\rho}{4} \\
& \geq \w{1}{2}{3}(p_1)+ \w{2}{3}{3}(p_2) - \left(Q(e_3^y) + Q(e_1^y) - Q(e_2^y)\right) -\dfrac{3\rho}{4}\\
&> \w{1}{2}{3}(p_1)+ \w{2}{3}{3}(p_2) - \rho -\dfrac{3\rho}{4} \\
&> 2\rho - \rho -\dfrac{3\rho}{4} =\dfrac{ \rho}{4} > 0. \qedhere
\end{align*}
\end{proof}

\begin{thm}\label{cornertrees}
For almost every $s=(s_1,s_2,s_3) \in S^3 = S \times S \times S$, $\overline{E}_s^{-1}(\Delta^3) = \bmodsp$ is a compact manifold with corners of dimension $|p_0|-|p_1|-|p_2|$ with $i-stratum$ $\bmodsp_i = \overline{X}_i \cap \overline{E}_s^{-1}(\Delta^3)$ given by trees with a total of $i$ breaks on the tree edges. 
\end{thm}
\begin{proof}
We transversely cut out a smooth manifold with corners from $\overline{X}$, making use of extensions of Transversality and Preimage Theorems for manifolds with corners; see Theorems \ref{preimage} and \ref{transcorners}.

We use Theorem \ref{transcorners} to show that for almost every $s=(s_1,s_2,s_3) \in S \times S \times S$, $\partial_\ell \overline{E}_s = \overline{E}_s |_{\overline{X}_\ell}$ is transversal to the diagonal $\Delta^3 \subset P_3 \times P_3 \times P_3$. This will imply, by Theorem \ref{preimage}, $\overline{E}_s^{-1}(\Delta^3)$ is a smooth submanifold with corners of $\overline{X}$ whose $\ell$-stratum $\overline{E}_s^{-1}(\Delta^3)_\ell$ is $\overline{X}_\ell \cap \overline{E}_s^{-1}(\Delta^3)$.

To use Theorem \ref{transcorners}, we need to show that $\partial_i \overline{E}$ is transversal to $\Delta^3$ for all strata of $\overline{X}$. Fix a trajectory sequence $\overline{\gamma}$ in $\Mfrob{1}{2}{3}{p_1}, \Mfrob{2}{3}{3}{p_2},$ or $\overline{\M}_{1,3;3}(\{\w{1}{3}{3} > \dfrac{\rho}{8}\}, p_0)$. 

For $M=\R^n$, the map $\overline{E}_{i,j;3}$ restricted to $\overline{\gamma}$ is the translation $s \mapsto x +s$ where $x = \ev{i}{j}{3}{\pm}(\overline{\gamma})$ and so is a submersion, and since $\overline{\gamma}$ was arbitrary, $E_{i,j;3}$ restricted to any strata (which is exactly the map $\partial^i \overline{E}_{i,j;3}$) of $\Mfrob{i}{j}{3}{p_i}$ or $\overline{\M}_{1,3;3}(\{\w{1}{3}{3} > \dfrac{\rho}{8}\}, p_0)$ is a submersion. 

The case where $M$ is closed is similar. For fixed $\overline{\gamma}$ as above, the map $\overline{E}_{i,j;3}$ sends $s$ to $ \pi(x +s)$ and is a composition of a translation with $\pi$, which was chosen through the $\epsilon$-Neighborhood Theorem to be a submersion (see Definition \ref{pert}). Since a restriction to one trajectory sequence is a submersion, a restriction to any stratum will be as well.

Thus, fixing a triple of trajectories $\hat{\gamma} \coloneqq \left( \overline{\gamma}_1, \overline{\gamma}_2, \overline{\gamma}_3 \right) \in \overline{X}$, note that, no matter which stratum this triple lives in, the map $E_{\hat{\gamma}}: S \times S \times S \to P_3 \times P_3 \times P_3$ is a product of submersions of the ball $S$. Thus, any restriction of $E$ to any strata of $\overline{X}$ is transversal to any submanifold of $(P_3)^3$, which shows that $\partial_i \overline{E} \pitchfork \Delta^3$.

Thus, for almost every $s \in S \times S \times S$, $\bmodsp$ will be a smooth manifold with corners whose codimension in $\overline{X}$ equals the codimension of $\Delta^3$ in $(P_3)^3$. 
From this, Remark \ref{smothstr}, and Equation \ref{eqn:index}, we can calculate:
\begin{align*}
&\dim (\bmodsp)\\
&=\dim (\modsp)\\
 &= \dim(X) - \left( \dim ((P_3)^3) - \dim(\Delta^3) \right) \\
&= \dim (\Mfro{1}{2}{3}{p_1}) + \dim (\Mfro{2}{3}{3}{p_2}) + \dim (\Mto{1}{3}{3}{p_0}) - 2(n+3N) \\
&= \dim (W^-_{p_1}(\w{1}{2}{3})) + \dim (W^-_{p_2}(\w{2}{3}{3})) + \dim (W^+_{p_0}(\w{1}{3}{3}))- 2(n+3N) \\
&= (n+3N) - \ind_{\w{1}{2}{3}}(p_1) + (n+3N) - \ind_{\w{2}{3}{3}}(p_2) + \ind_{\w{1}{3}{3}}(p_0) - 2(n+3N) \\
&= -(|p_1| + N) - (|p_2| + N) + (|p_0| + 2N) \\
&= |p_0| -  |p_1| - |p_2|.
\end{align*}


It remains to show that $\bmodsp$ is compact. This would be immediate if $\overline{X}$ were compact, as $\Delta^3$ is closed in $(P_3)^3$ and so $\overline{E}_s^{-1}(\Delta^3)$ is closed in $\overline{X}$. We will argue, instead, that trajectory sequences that cause $\overline{X}$ to be noncompact do not show up in trees in $\bmodsp$. We know that broken half-infinite trajectory spaces for Morse-Smale pairs on a closed manifold are compact from \cite[Theorem 2.3]{wehrheim:morse}, so issues of noncompactness in 
$$\overline{X} = \overline{\M}_{1,2;3}(p_1,P_3) \times \overline{\M}_{2,3;3}(p_2,P_3) \times \overline{\M}_{1,3;3}\left(\left\lbrace\w{1}{3}{3} > \dfrac{\rho}{8}\right\rbrace, p_0\right)$$ stem from the noncompactness of $P_3$ and $\{\w{1}{3}{3} > \dfrac{\rho}{8}\}$. In particular, the spaces $\Mfrob{i}{j}{3}{p_i}$ could contain a sequence of trajectories whose finite ends (i.e., images of $\ev{i}{j}{3}{+}$) diverge. Similarly, there could be sequence in $\overline{\M}_{1,3;3}(\{\w{1}{3}{3} > \dfrac{\rho}{8}\}, p_0)$ whose limit has finite end, given by $\ev{i}{j}{3}{-}$, in the level set $\{\w{1}{3}{3} = \dfrac{\rho}{8}\}$. 

As $\bmodsp$ is a metric space, to show that it is compact it suffices to prove sequential compactness. Suppose $\Gamma_n = (\overline{\gamma}_1, \overline{\gamma}_2, \overline{\gamma}_3)_n$ is a sequence of trees in $\bmodsp$. The same proofs of Lemma \ref{lem:cpct-set} and \ref{awayfrom0} show that, for all $n$, $\Gamma_n \subset K_s$ and $\w{1}{3}{3}((\gamma_3)_n) > \dfrac{\rho}{4} > \dfrac{\rho}{8}$. 

With these bounds, the convergence of a subsequence of $\Gamma_n$ follows as in the proof of \cite[Theorem 2.3]{wehrheim:morse} and \cite[Proposition 3]{bhsmale}. This was shown by defining a continuous reparametrization of the images of trajectories in the sequence with bounded derivatives on the complements of neighborhoods of critical points. This implies the equicontinuity of these reparametrizations, which, by the Arzel\`{a}-Ascoli Theorem, gives a convergent subsequence.
\end{proof}

\begin{defn}\label{modsp}
We may describe the 0-stratum $$\modsp \coloneqq \bmodsp_0$$ in the following way:

Given a generating family $F:M \times \R^N \to \R$, pick metrics $g_{i,j;3}$ as in Definition \ref{EMS}. Let $S$ be a perturbation ball as in Definition \ref{pert} and form the $E_{i,j;3}$ maps as in Definition \ref{Eijbroken}. Theorem \ref{cornertrees} implies that we can choose $s=(s_1,s_2,s_3) \in S \times S \times S$ so that the following set is a smooth manifold.
$$\mathcal{M}(p_1,p_2; p_0 | s) = 
  \left\lbrace 	(\gamma_1,\gamma_2,\gamma_3)  \mathrel{}\middle|  \mathrel{} \begin{aligned} 
  & \gamma_1: \left( -\infty, 0 \right] \rightarrow M \times \R^N \times \R^N \times \R^N, \\
  & \gamma_2: \left( -\infty, 0 \right] \rightarrow  M \times \R^N \times \R^N \times \R^N, \\ 
  & \gamma_3: \left[0, \infty \right) \rightarrow  M \times \R^N \times \R^N \times \R^N, \\ 
  & \dfrac{d \gamma_1}{dt} = \nabla_{g_{1,2;3}} \w{1}{2}{3}, \dfrac{d \gamma_2}{dt} = \nabla_{g_{2,3;3}} \w{2}{3}{3}, \\ &\dfrac{d \gamma_3}{dt} = \nabla_{g_{1,3;3}} \w{1}{3}{3}, \\
  & E_{1,2;3}(\gamma_1,s_1)= E_{2,3;3}(\gamma_2,s_2) = E_{1,3;3}(\gamma_3,s_3),\\
  & \lim_{t \to -\infty} \gamma_1(t) = p_1, \lim_{t \to -\infty} \gamma_2(t) = p_2, \\
  & \lim_{t \to \infty} \gamma_3(t) = p_0 \end{aligned} \right\rbrace  $$ 
\end{defn}

We apply Theorem \ref{cornertrees} to see that a 1-dimensional $\modsp$ has a natural compactification through the addition of trees with once-broken edges, see Figure \ref{brokentrees}.

\begin{cor} \label{compactification}

Given $p_1, p_2, p_0$ with $|p_0| -  |p_1| - |p_2| = 1$, $\modsp$ can be compactified to a 1-manifold $\bmodsp$ with boundary 
$$\begin{aligned}
 \partial \bmodsp  \coloneqq &\bigcup_{p_1'} \Mtofro{1}{2}{p_1}{p_1'} \times \mathcal{M}(p_1',p_2; p_0 | s) \\
 & \bigcup_{p_2'} \Mtofro{2}{3}{p_2}{p_2'} \times \mathcal{M}(p_1,p_2'; p_0 | s) \\
& \bigcup_{p_0'} \mathcal{M}(p_1,p_2;p_0' | s) \times \Mtofro{1}{3}{p_0'}{p_0},
\end{aligned}$$
where the unions are taken over $p_1' \in  \V{|p_1|+1},$ $p_2' \in  \V{|p_2|+1},$ and $p_0' \in  \V{|p_0|-1}$.

%
\end{cor}

\begin{figure}[htb] 
\includegraphics[width=1 \textwidth]{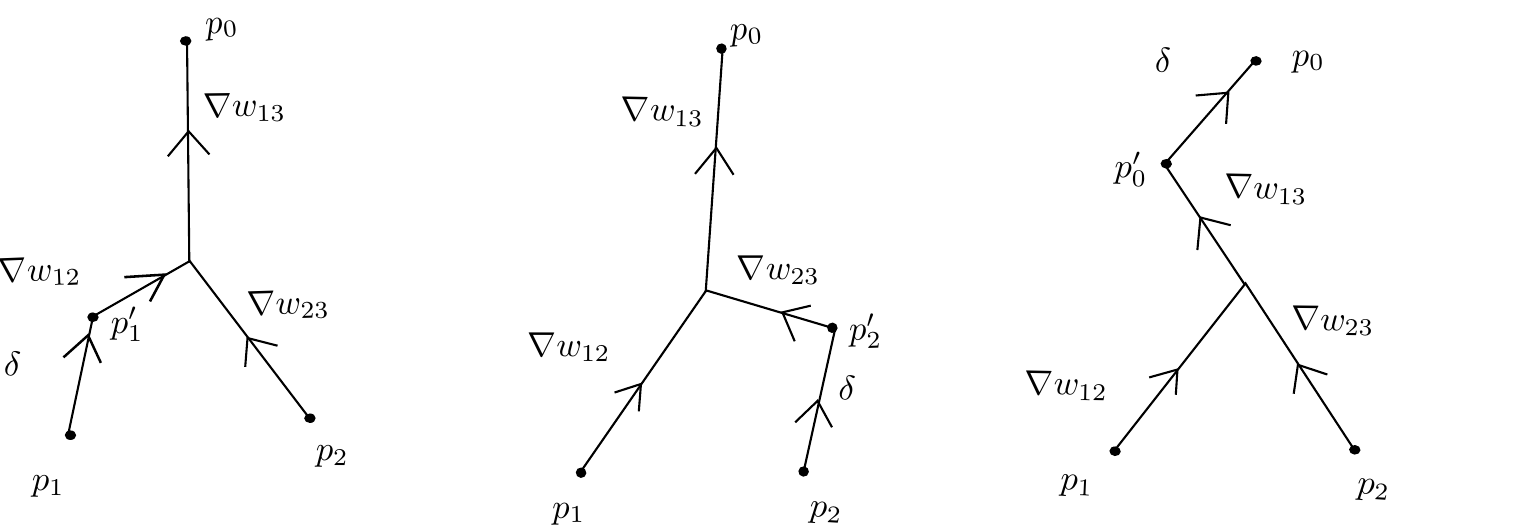} 
\caption{Elements in $\partial \modsp$ for $s=(0,0,0)$.} 
\label{brokentrees}
\end{figure}

\section{Product Structure}\label{sec:product}


Now that we have defined the manifold $\modsp$ and its compactification $\bmodsp$, a manifold with corners, we define a map on the cochain level of $GH^*(F)$ by counting isolated trees in $\modsp$ and show it to be a cochain map by considering a one-dimensional $\bmodsp$. Note that, for $p_1,p_2,p_0 \in \crit_+(w)$, Theorem \ref{brokentrees} implies that isolated trees in $\modsp$ satisfy 

$$|p_0| = |p_1| + |p_2| $$

\begin{defn}\label{productmap}
Given a generating family $F:M\times \R^N \to \R$, we define a map
 $$m_2: \V{i} \otimes \V{j}\rightarrow \V{i+j} $$
as follows: for critical points $p_1, p_2 \in \crit_+(w)$, define

$$m_2(p_1 \otimes p_2) = \sum ( \#_{\Z_2} \modsp ) \cdot p_0 $$ 
where the sum is taken over $p_0 \in \crit_+(w)$ 
such that $|p_0| = |p_1| + |p_2| $. Extend the product bilinearly over the tensor product.
\end{defn}


The following lemma shows that $m_2$ descends to a map on cohomology:

\begin{lem} The map $$m_2: \V{i} \otimes \V{j}\rightarrow \V{i+j} $$ is a cochain map, i.e., the following diagram commutes: 

 \begin{equation} \label{eqn:cochain}
      \xymatrix{
       \V{} \otimes \V{} \ar[d]_{\delta \otimes 1 + 1 \otimes \delta}  \ar[r]^{\quad \quad m_2} &\V{} \ar[d]^{\delta} \\
       \V{} \otimes \V{} \ar[r]^{\quad \quad m_2} &\V{}.
      }
    \end{equation}

\end{lem}

\begin{proof}
Consider a $1$-dimensional moduli space of flow trees, $\modsp$.  Theorem \ref{brokentrees} shows that such a space occurs when 
$$|p_0| = |p_1| + |p_2| +1 ,$$
so let $p_1 \in \V{i}$, $p_2 \in \V{j},$ and $p_0 \in \V{i+j+1}$.

Corollary \ref{compactification} gives an expression for $\partial \modsp$. In particular, the boundary of $\modsp$ consists of isolated trees with a single broken edge. After compactification, $\modsp$ is a compact $1$-manifold, so its boundary contains an even number of points. Thus, a $\Z_2$ count of both sides of the expression for $\partial \modsp$ gives us:
\begin{align}
\begin{split}
0 &= \sum_{p_1'} \#_{\Z_2} \Mtofro{1}{2}{p_1}{p_1'} \cdot  \#_{\Z_2} \M(p_1',p_2;p_0 | s) \\
&+ \sum_{p_2'} \#_{\Z_2} \Mtofro{2}{3}{p_2}{p_2'} \cdot  \#_{\Z_2} \M(p_1,p_2';p_0 | s) \\
&+ \sum_{p_0'}  \#_{\Z_2} \M(p_1,p_2;p_0' | s)  \cdot  \#_{\Z_2} \Mtofro{1}{3}{p_0'}{p_0}.
\end{split} \label{boundarycount}
\end{align}

This now implies the cochain map condition
$$ m_2(\delta p_1 \otimes p_2) + m_2(p_1 \otimes \delta p_2) = \delta m_2(p_1 \otimes p_2). $$ 
This follows since the terms on the right hand side of Equation \ref{boundarycount} are exactly the coefficients of the three terms in the cochain map condition.

As an example, consider the term $m_2(\delta p_1 \otimes p_2)$:

$$ m_2(\delta p_1 \otimes p_2) = \sum_{p_0} \#_{\Z_2} \M (\delta p_1,p_2,p_0 | s) \cdot p_0 $$
$$= \sum_{p_0} \#_{\Z_2} \M(\sum_{p_1'} \#_{\Z_2} \Mtofro{1}{2}{p_1}{p_1'} \cdot p_1',p_2,p_0 | s) \cdot p_0 $$

$$= \sum_{p_0} ( \sum_{p_1'}  \#_{\Z_2} \Mtofro{1}{2}{p_1}{p_1'} \cdot \#_{\Z_2} \M(p_1',p_2,p_0 | s) ) \cdot p_0.  $$

The other two terms follow similarly, which shows that $m_2$ is a cochain map, as desired.
\end{proof}

\begin{cor}\label{cor:product}
Given a generating family $F:M\times \R^N \rightarrow \R$, there is a product map on Generating Family Cohomology
$$\mu_2: GH^{i} (F) \otimes GH^{j} (F) \rightarrow GH^{i+j}(F). $$
\end{cor}

\section{Invariance with respect to Equivalences of $F$}\label{secinvar}
Recall from Subsection \ref{ssec:gf-back} that there is a notion of equivalence $\sim$ of generating families for a Legendrian submanifold $\Lambda \subset J^1M$. Lemma \ref{lem:cohom-equiv} shows that $GH^*(F)$ is invariant under $\sim$. In this section, we show that the product is unchanged under $\sim$ as well.

This amounts to showing that, when $\widehat{F}$ is obtained from $F$ by stabilization or fiber-preserving diffeomorphism 
resulting in isomorphisms $GH^*(F) \to GH^*(\widehat{F})$, the following diagram commutes:
\begin{equation} \label{eqn:productequiv}
      \xymatrix{
       GH^*(F) \otimes GH^*(F) \ar[d]_{\cong}  \ar[r]^{ \quad\quad \mu_2} &GH^*(F) \ar[d]^{\cong} \\
       GH^*(\widehat{F}) \otimes GH^*(\widehat{F}) \ar[r]^{\quad \quad \widehat{\mu}_2} &GH^*(\widehat{F}).
      }
    \end{equation}

\subsection{Fiber Preserving Diffeomorphism}\label{fibpresdif}
In this subsection, we analyze how the product is affected when we pre-compose our generating family $F: M \times \R^N \to \R$ with a fiber-preserving diffeomorphism $\Phi: M \times \R^N \to M \times \R^N$. Recall from Subsection \ref{ssec:gf-back} that, by definition, $\Phi(x,e) = (x, \phi_x(e))$ for a smooth family of diffeomorphisms $\phi_x: \R^N \to \R^N.$ As in Lemma \ref{lem-fpd},  we consider diffeomorphisms $\phi$ that are isometries outside the compact set $K_E$ since our setup of gradient flow uses metrics that are Euclidean outside $K$; see Definitions \ref{wEMS} and \ref{EMS}.

\begin{rem}\label{phidiag}
We may extend $\Phi$ naturally to a diffeomorphism on $P_3 = M \times R^{3N}$: abusing notation, let $\Phi: P_3 \to P_3$ be defined as 
$$(x, e_1, e_2, e_3) \mapsto (x, \phi_x(e_1), \phi_x(e_2), \phi_x(e_3)).$$ 
\end{rem}

\begin{lem}
For $\Delta^3 \subset (P_3)^3,$ $\Phi^3(\Delta^3) = \Delta^3$.
\end{lem}
\begin{proof}
This is not a hard fact, but we write out the proof to recall the space $\Delta^3$. We use coordinates $(x,e_1,e_2,e_3)$ on $P_3 = M \times \R^{3N}$, so we have natural coordinates 
$$\left( x_1, e_{11}, e_{21}, e_{31}, x_2, e_{12}, e_{22}, e_{32}, x_3, e_{13}, e_{23}, e_{33}\right). $$ With these coordinates, $\Delta^3$ is the submanifold in which $x_1 = x_2 = x_3$ and $e_{i1} = e_{i2} =e_{i3}$ for $i = 1, 2, 3$, which is preserved under $\Phi$.
\end{proof}

Lemma \ref{gradvfcorres} showed that critical points and gradient trajectories correspond under diffeomorphism. This induces diffeomorphisms of the stable and unstable manifolds which gives a diffeomorphism $\widetilde{X} \cong X$.

\begin{equation} \label{eqn:fpdpert}
      \xymatrix{
       X \times S^3 \ar[d]  \ar[r]^{E} &(P_3)^3 \ar[d]^{\cong} &\Delta^3 \ar@{_{(}->}[l] \ar@{|->}[d]\\
       \widetilde{X} \times \widetilde{S}^3 \ar[r]^{E} &(P_3)^3 &\Delta^3 \ar@{_{(}->}[l] }
    \end{equation}

\begin{lem}
Suppose $\widetilde{F}$ is obtained from $F$ through fiber-preserving diffeomorphism. Let $s= (s_1,s_2,s_3) \in S \times S \times S$ and $p_i \in \V{}$ so that $\modsp$ is a 0-dimensional manifold. Then, for corresponding $\widetilde{p}_i \in C(\widetilde{F})$, there exists an $\widetilde{s}= (\widetilde{s}_1, \widetilde{s}_2, \widetilde{s}_3) \in \widetilde{S} \times \widetilde{S} \times \widetilde{S}$ for some $\widetilde{\delta}$ ball $\widetilde{S}$ such that $\widetilde{\M}(\widetilde{p}_1,\widetilde{p}_2;\widetilde{p}_0 | \widetilde{s})$ is in bijection with $\modsp$.
\end{lem}

\begin{proof}
Given a tree $\Gamma = \{\gamma_1, \gamma_2, \gamma_3\} \in \modsp$, there exists a $(y,y,y) \in \Delta^3 \subset (P_3)^3$ such that $E(\Gamma) = (y,y,y)$, i.e., $E_{1,2;3}(\gamma_1) = E_{2,3;3}(\gamma_2) = E_{1,3;3} (\gamma_3) = y$.

In the case that $M=\R^m$, this means that $y = \gamma_1(0) +s_1 = \gamma_2(0) +s_2 = \gamma_3(0) + s_3$. From Lemma \ref{gradvfcorres}, we know that there are corresponding $\widetilde{\gamma}_i$. Remark \ref{phidiag} gives us an element $\widetilde{y} = \Phi(y) \in \Delta^3$. Thus, there is a unique way to pick $\widetilde{s}_i$ so that $\widetilde{y} = \widetilde{\gamma}_1(0) +\widetilde{s}_1 = \widetilde{\gamma}_2(0) +\widetilde{s}_2 = \widetilde{\gamma}_3(0) + \widetilde{s}_3$.

For this $\widetilde{s} = (\widetilde{s}_1, \widetilde{s}_2, \widetilde{s}_3)$, we have that $E_{\widetilde{s}} \pitchfork \Delta^3$ and $\widetilde{s}_i \in \widetilde{S}$ for each $i$, that is, $|\widetilde{s}_i| < \widetilde{\delta}$. Here, $\widetilde{\delta}$ is such that for all $y_1, y_2 \in \widetilde{K}$, $|y_1-y_2| < \widetilde{\delta}$ implies that $|(\w{i}{j}{3}\circ \Phi)(y_1) - (\w{i}{j}{3}\circ \Phi)(y_2)| < \rho/4$, where $\rho$ is the least positive critical value of $w$, which is the same as the least positive critical value of $w \circ \Phi$.
\end{proof}

\begin{rem}
We see here why it was necessary to define the extended difference functions in 
Definition \ref{extdiff} as ``quadratic-like" stabilizations of the difference function $w$. In particular, if $Q(e_k) = e_k^2,$, then $$(\w{i}{j}{3} \circ \Phi)(x,e_1,e_2,e_3) = (w \circ \Phi)(x,e_i,e_j) \pm (\phi_x(e_k))^2, $$
and $(\phi_x(e_k))^2$ is not necessarily a quadratic form. It is however, a function with only one critical point with preserved index and preserved critical value. In short, if $Q$ is quadratic-like as in Definition \ref{extdiff}, $Q \circ \phi_x$ is as well.
\end{rem}

\begin{cor}\label{cor:fpd}
Suppose $F:M\times \R^N \rightarrow \R$ is altered by a fiber-preserving diffeomorphism $\Phi: M\times \R^N \rightarrow M\times \R^N$, where $\Phi(x, e) = (x, \phi_x(e))$ for some diffeomorphisms $\phi_x: \R^N \rightarrow \R^N$ resulting in $\widetilde{F}= F \circ \Phi$. Then the following diagram commutes:
\begin{equation} \label{eqn:fiberpresdiffeo}
      \xymatrix{
       GH^*(F) \otimes GH^*(F) \ar[d]_{\cong}  \ar[r]^{ \quad\quad \mu_2} &GH^*(F) \ar[d]^{\cong} \\
       GH^*(\widetilde{F}) \otimes GH^*(\widetilde{F}) \ar[r]^{\quad \quad \widetilde{\mu_2}} &GH^*(\widetilde{F}).
      }
    \end{equation}
\end{cor}

\subsection{Stabilization}

Given a generating family $F: M \times \R^N \to \R$, define $F^{\pm}: M \times \R^N \times \R \to \R$ by $F^{\pm}(x,e,e') = F(x,e) \pm (e')^2$. To show invariance under stabilization, it suffices to show that the diagram \ref{eqn:productequiv} commutes for $\widehat{F}=F^\pm$. 

Since we showed that the product is invariant under fiber-preserving diffeomorphism in the previous subsection, for simplicity of the following argument, we may precompose with such a diffeomorphism and assume that the stabilization of the extended difference function is quadratic rather than quadratic-like.
With this assumption, observe that we can write out new extended difference functions $$\w{i}{j}{3}^\pm: M \times \R^{3N} \times R^3 \to \R$$ from $F^\pm$ in terms of stabilizations of $F$:
%
%
\begin{align*}
 \w{1}{2}{3}^\pm(x,e,e')&= F(x,e_1) \pm (e_1')^2 - F(x, e_2) \mp (e_2')^2 + e_3^2 +(e_3')^2 \\
 \w{2}{3}{3}^\pm(x,e,e')&= F(x,e_2) \pm (e_2')^2 - F(x, e_3) \mp (e_3')^2 + e_1^2 +(e_1')^2 \\
 \w{1}{3}{3}^\pm(x,e,e')&= F(x,e_1) \pm (e_1')^2 - F(x, e_3) \mp (e_3')^2 - e_2^2 -(e_2')^2,
\end{align*}
We may express these new stabilized extended difference functions as 
$$\w{i}{j}{3}^{\pm}(x,e_1,e_1',e_2,,e_2',e_3,e_3') =  \w{i}{j}{3}(x,e_1,e_2,e_3) + Q_{i,j;3}^\pm(e_1',e_2',e_3') $$
for different nondegenerate quadratic functions $Q_{i,j;3}^\pm: \R^3 \to \R$. 


\begin{lem}\label{stabcritpts}
Given a generating family $F$ and $F^\pm$ as above, if $p \in \V{}$, then there are corresponding critical points $p^\pm \in C(F^\pm)$ with the same critical value and $|p^\pm| = |p| +1$.
\end{lem}

\begin{rem}
Note that, by construction, this correspondence passes to the extended difference functions and bijections in Lemma \ref{iotabij}: if $p \in \Crit(\w{i}{j}{3})$, then there is a corresponding critical point $p^\pm \in \Crit(\w{i}{j}{3}^{\pm})$ whose primed coordinates are 0. Hence, $\w{i}{j}{3}(p) = \w{i}{j}{3}^{\pm}(p^\pm)$. The index increases by 1 if $j-i =1$ and 2 if $j-i= 2$.
\end{rem}

The gradient trajectories we are interested in now live in $P_3 \times \R^3$ rather than $P_3$. To study trajectories, we equip $P_3 \times \R^3$ with split metrics $g_{i,j;3}' = g_{i,j;3} + g_0$ where $g_{i,j;3}$ is a metric on $P_3$  as in Definition \ref{EMS} and $g_0$ is the standard Riemannian metric on $\R^3$. Such a metric, if generic, facilitates comparison of gradient trajectories of the stabilized extended difference functions to those before stabilization.

\begin{lem}
If $g_{i,j;3}$ is a metric of the form in Definition \ref{EMS}; that is, if $g_{i,j;3} = g_w + g_Q$, then $g_{i,j;3}'=g_{w^\pm} + g_{Q'}$ for $g_{w^\pm} \in \mathcal{G}_{F^\pm}$ and $g_{Q'} \in \mathcal{G}_{Q'}$. Here, $\mathcal{G}_{F^\pm}$ and $\mathcal{G}_{Q'}$ are the metric sets defined in Definition \ref{EMS} for the stabilized generating family $F^\pm: M \times \R^{N+1}$ and corresponding quadratic form $Q':\R^{N+1} \to \R$ so that $\w{i}{j}{3}^\pm = w^\pm + Q'$.
\end{lem}  

\begin{proof}
The only non-immediate condition to check is the Smale condition, but since $w^\pm (x, e_1, e_1', e_2, e_2') = w(x,e_1,e_2) \pm (e_1')^2 \mp (e_2')^2$, the techniques in the proof of Proposition \ref{Smaleprop} show that these metrics will ensure the Smale condition.
\end{proof}
%

\begin{rem}\label{Xstab}
With this choice of metrics $g_{i,j;3}$, the below relations between the unstable/stable manifolds hold, where $p_i^\pm \in C(F^\pm)$ denotes the corresponding critical point to $p_i \in \V{}$, see \ref{stabcritpts}. Note that we are abusing notation as promised in Remark \ref{abusenot}. The first diffeomorphism, as noted by Remark \ref{smothstr}, is due to the fact that the Morse trajectory spaces inherit their smooth structures from the unstable and stable manifolds.

\begin{align*}
\M^+_{1,2;3}(p_1^+, P_3 \times \R^3) \cong W^-_{p_1^+}(\w{1}{2}{3}^+) &\cong W^-_{p_1}(\w{1}{2}{3}) \times \R_{e_1'} \times \{ 0 \}_{e_2'} \times \R_{e_3'} \\
\M^+_{2,3;3}(p_2^+, P_3 \times \R^3) \cong W^-_{p_2^+}(\w{2}{3}{3}^+) &\cong W^-_{p_2}(\w{2}{3}{3}) \times \R_{e_1'} \times \R_{e_2'} \times \{ 0 \}_{e_3'} \\
\M^+_{1,3;3}(P_3 \times \R^3, p_0^+) \cong W^+_{p_0^+}(\w{1}{3}{3}^+) &\cong W^+_{p_0}(\w{1}{3}{3}) \times \{ 0 \}_{e_1'} \times \R_{e_2'} \times \R_{e_3'} \\
\M^-_{1,2;3}(p_1^-, P_3 \times \R^3) \cong W^-_{p_1^-}(\w{1}{2}{3}^-) &\cong W^-_{p_1}(\w{1}{2}{3}) \times \{0\}_{e_1'} \times \R_{e_2'} \times \R_{e_3'} \\
\M^-_{2,3;3}(p_2^-, P_3 \times \R^3) \cong W^-_{p_2^-}(\w{2}{3}{3}^-) &\cong W^-_{p_2}(\w{2}{3}{3}) \times \R_{e_1'} \times \{ 0 \}_{e_2'} \times \R_{e_3'} \\
\M^-_{1,3;3}(P_3 \times \R^3, p_0^-) \cong W^+_{p_0^-}(\w{1}{3}{3}^-) &\cong W^+_{p_0}(\w{1}{3}{3}) \times \R_{e_1'} \times \R_{e_2'} \times \{ 0 \}_{e_3'} 
\end{align*}
\end{rem}

Remark \ref{Xstab} tells us how the space $X = \Mfro{1}{2}{3}{p_1} \times \Mfro{2}{3}{3}{p_2} \times \Mto{1}{3}{3}{p_0}$ in which our moduli space of flow trees lives, compares to the space $$X^\pm = \M^\pm_{1,2;3}(p_1^\pm, P_3 \times \R^3) \times \M^\pm_{2,3;3}(p_2^\pm, P_3 \times \R^3) \times \M^\pm_{1,3;3}(P_3 \times \R^3, p_0^\pm)$$ obtaining using flows from $F^\pm$ and $g$. 

It remains to check transversality of perturbed evaluation at endpoints maps with the diagonal $\widehat{\Delta}^3 \cong \Delta^3 \times \Delta_{\R^3} \subset (P_3 \times \R^3)^3$ persists, and that the resulting preimage, the moduli space of flow trees, is diffeomorphic to the preimage from before. Given a perturbation $s=(s_1,s_2,s_3) \in S^3$, we claim that $(s,0) \in \widehat{S}^3$ achieves transversality, where $\widehat{S}$ is the perturbation ball in $P_3 \times \R^3$ as defined in Definition \ref{pert}.

\begin{equation} \label{eqn:stabpert}
         \xymatrix{
      &\Delta^3 \ar@{_{(}->}[d] \ar@{.>}[ld] \\
       X  \ar[d]_{}  \ar[r]^{E_s} &(P_3)^3 \ar[d] \\
       X^\pm \ar[r]^{E_{(s,0)}^\pm} &(P_3 \times \R^3)^3 \\
      &\widehat{\Delta}^3 \ar@{^{(}->}[u] \ar@{.>}[lu] }
    \end{equation}
    
\begin{lem}
Given $F: M \times \R^N \to \R$ and $F^\pm:M \times \R^N \times \R \to \R$, let $p_1, p_2, p_0 \in \V{}$ have corresponding critical points  $p_1^\pm, p_2^\pm, p_0^\pm \in C(F^\pm)$ (see Lemma \ref{stabcritpts}). Then if $E_s \pitchfork \Delta^3 $ then $\widehat{E}_{(s,0)} \pitchfork \widehat{\Delta}^3$.
\end{lem}

\begin{proof}
Our choice of metric and perturbation reduces this question to a elementary differential topology one. We wish to show the following in $\R^9$:
$$\left(W^-_0(Q_{1,2;3}^\pm) \times W^-_0(Q_{2,3;3}^\pm) \times W^+_0(Q_{1,3;3}^\pm) \right) \pitchfork \Delta_{\R^3}^3.$$
Remark \ref{Xstab} tells us what these stable and unstable manifolds are. Since the product of these manifolds in both the $+$ and $-$ case is 6-dimensional and $\Delta_{\R^3}^3$ is 3-dimensional, the result follows because 
$$\left(W^-_0(Q_{1,2;3}^\pm) \times W^-_0(Q_{2,3;3}^\pm) \times W^+_0(Q_{1,3;3}^\pm) \right) \cap \Delta_{\R^3}^3 = \{0\}, $$ so their tangent spaces must span $\R^9$.
\end{proof}
    
With our choice of split metric and no extra perturbation, the following lemma shows that the moduli space of flow trees that define out product splits into a space of trees defined through the original generating family $F$ and ``constant" trees, that is, three constant trajectories at $0 \in \R^3$.

\begin{lem}
Given $F, F^\pm, p_i \in \V{},$ and $p_i^\pm \in C(F^\pm)$, $\widehat{\M}(p_1^\pm, p_2^\pm; p_0^\pm \mid (s,0))$ is diffeomorphic to  $\modsp.$
\end{lem}

\begin{proof}
As $\widehat{\M}(p_1^\pm, p_2^\pm; p_0^\pm \mid (s,0)) = \widehat{E}_{(s,0)}^{-1}(\widehat{\Delta}^3)$ by definition, our setup shows that we may split this preimage $$\widehat{E}_{(s,0)}^{-1}(\widehat{\Delta}^3) \cong E_s^{-1}(\Delta^3) \times E_0^{-1}(\Delta_{\R^3}). $$ As $E_s^{-1}(\Delta^3) = \modsp$, we consider $E_0^{-1}(\Delta_{\R^3})$. This space consists of triples of trajectories $\{\gamma_1, \gamma_2, \gamma_3\}$ with $\gamma_1, \gamma_2: (-\infty, 0] \to \R^3$ and $\gamma_3:[0, \infty) \to \R^3$. The trajectory $\gamma_1$ flows from $0 \in \R^3$ and follows $\nabla_{g_0} Q_{1,2;3}^\pm$, and so, by Remark \ref{Xstab}, is contained in the $(e_1', e_3')$-plane in the $+$ case and the $(e_2', e_3')$-plane in the $-$ case. Similarly, $\gamma_2$ flows from $0 \in \R^3$ and follows $\nabla_{g_0} Q_{2,3;3}^\pm$, and so is contained in the $(e_1', e_2')$-plane in the $+$ case and the $(e_1', e_3')$-plane in the $-$ case. Thus, in the $+$ case, $\gamma_1$ and $\gamma_2$ intersect along the $e_1'$-axis; in the $-$ case, they intersect along the $e_3'$-axis. In either case, $\gamma_3$ intersects the intersection of $\gamma_1$ and $\gamma_2$ and flows to $0 \in \R^3$ along $\nabla_{g_0} Q_{1,3;3}^\pm$. In both the $+$ and $-$ case, however, $\gamma_3$ trajectory will only intersect the $e_1'$-axis ($e_3'$-axis) at $0$, and thus has to be the constant trajectory. This implies that both $\gamma_1$ and $\gamma_2$ never flowed off of the critical point 0, and are also constant trajectories.
\end{proof}

\begin{cor}\label{cor:stab}
If $F:M\times \R^N \rightarrow \R$ is altered by a positive or negative stabilization resulting in $\widehat{F}:M\times \R^N \times \R \rightarrow \R$ then the following diagram commutes:
\begin{equation} \label{eqn:stabilization}
      \xymatrix{
       GH^*(F)  \otimes GH^*(F) \ar[d]_{\cong}  \ar[r]^{\quad  \quad \mu_2} &GH^*(F) \ar[d]^{\cong} \\
       GH^*(F^\pm) \otimes GH^*(F^\pm) \ar[r]^{\quad \quad \mu_2^\pm} &GH^*(F^\pm).
      }
    \end{equation}
\end{cor}

\section{Invariance under Legendrian isotopy}\label{sec:inviso}

In this section, we study the product as the underlying Legendrian $\Lambda$ undergoes a Legendrian isotopy. In particular, suppose we have a Legendrian isotopy $\Lambda^t$ with $t \in [0,1]$, and suppose $\Lambda^0$ has a generating family. From the Persistence of Legendrian Generating Families (see Proposition \ref{prop:leg-persist}), the isotopy lifts to a smooth path of generating families $F^t$ for $\Lambda^t$. In Section \ref{sec:gh-legendrian}, we constructed a chain map that induces an isomorphism between $GH^*(F^0) \to GH^*(F^1)$ (Corollary \ref{cor:end-naturality}).


Given $F^0$ and the resulting $F^1$ guaranteed by Proposition \ref{prop:leg-persist}, we may assume by stabilization that both are functions on $M \times \R^N$. We wish to compare the product from $F^0$ with the product from $F^1$. In particular, we wish to show that the following diagram commutes:
\begin{equation}\label{isotopygoal}
      \xymatrix{
       GH^*(F^0) \otimes GH^*(F^0) \ar[d]_{\cong}  \ar[r]^{ \quad\quad \mu_2^0} &GH^*(F^0) \ar[d]^{\cong} \\
       GH^*(F^1) \otimes GH^*(F^1) \ar[r]^{\quad \quad \mu_2^1} &GH^*(F^1).
      }
    \end{equation}
    
While we know there exists isomorphisms $GH^*(F^0) \to GH^*(F^1)$, for the vertical isomorphisms in the diagram in \ref{isotopygoal}, we construct maps that will be compatible with the product. To do this, we slightly alter the setup of the continuation map in Section \ref{sec:gh-legendrian} to produce three continuation maps using paths of extended difference functions. We then extend this idea to form a moduli space of ``continuation flow trees" on $P_3 \times I$ and define a map $K$ counting isolated spaces of such trees. Studying the compactification of a 1-dimensional space of the trees shows that the map $K$ defines a chain homotopy that induces the commutative diagram \ref{isotopygoal} on cohomology.
    
The first subsection of this section deals with the vertical isomorphisms in the above diagrams, while the second constructs ``continuation trees" that will define a chain homotopy that implies the commutativity of \ref{isotopygoal}.
    
\subsection{Continuation isomorphisms on $GH^*(F)$}

Given the path of linear-at-infinity generating families from $F^0: M \times \R^N \to \R$ to $F^1: M \times \R^N \to \R$, we wish to compare the product at time $t=0$ to the one at $t=1$. For $F_0$ and $F_1$, we constructed continuation maps from the path of difference functions $w^t: M \times \R^N \times \R^N \to \R$ such that $w^t(x, e_1, e_2) = F^t(x,e_1) - F^t(x,e_2)$.

To get continuation isomorphisms that are compatible with the product, we will constuct them on the paths of extended difference functions for $t \in [0,1]$, denoted $\w{i}{j}{3}^t: M \times \R^N \times \R^N \times \R^N \to \R$ defined as usual by:
\begin{align*}
& \w{1}{2}{3}^t(x,e_1,e_2,e_3)= F^t(x,e_1) - F^t(x, e_2) + Q^t(e_3) \\
& \w{2}{3}{3}^t(x,e_1,e_2,e_3)= F^t(x,e_2) - F^t(x, e_3) + Q^t(e_1) \\
& \w{1}{3}{3}^t(x,e_1,e_2,e_3)= F^t(x,e_1) - F^t(x, e_3) + Q^t(e_2)
\end{align*}
For each $\w{i}{j}{3}^t$, there is the corresponding non-linear support compact set $K^t$, which will vary smoothly with $t$.

Given $F^0$ and the resulting $F^1$, construct the resulting extended difference functions as above and pick metrics $g_{i,j;3}^0, g_{i,j;3}^1$ as in Definition \ref{EMS}. Then let $\Gamma_{i,j;3} = \{(w^t_{i,j;3}, g^t_{i,j;3}) \mid t \in \left[ 0,1 \right] \}$ be a path of the extended difference functions and metrics on $P_3$ that are standard  outside $K^t$ from $(\w{i}{j}{3}^0, g_{i,j;3}^0)$ to $(\w{i}{j}{3}^1, g_{i,j;3}^1)$.

For each these three paths we have a continuation map $\Phi_{i,j;3}: C^*(F^0) \to C^*(F^1)$ defined by counting isolated flow lines of the vector field $\nabla_{G_{i,j;3}}W_{i,j;3}$ on $(M \times \R^{3N}) \times I$ with 
\begin{align}\label{nabGWij}
\begin{split}
W_{i,j;3}(p,t) &= \w{i}{j}{3}^t(p) + \epsilon \left((1/2)t^2 - (1/4)t^4\right) \\
(G_{i,j;3})_{(p,t)} &= (g_{i,j;3}^t)_p + dt^2,
\end{split}
\end{align}
for $\epsilon > 0$ such that $\dfrac{\epsilon}{4} < \rho$, where $\rho$ is the least positive critical value of $w$.

As done in detail in Section \ref{sec:gh-legendrian}, these maps induce isomorphisms which we will denote by $\Phi^*_{i,j;3}$, with
$$\Phi^*_{i,j;3}: GH^*(F^0) \to GH^*(F^1),$$
and the arguments in Proposition \ref{metricinvrgh} show that this map does not depend on the path $F^t$ up to homotopy class.

\subsection{Continuation flow trees}
To get the commutative diagram in \ref{isotopygoal}, we construct a chain homotopy by defining a moduli space of ``continuation flow trees." The construction will be similar to that of $\modsp$. Now, our trees will live in $P_3 \times I$ rather than $P_3$ and we will require that the trees span $I$, i.e., flow along trajectories out of two critical points at $t=0$ and along a trajectory that limits to a critical point at $t=1$. We will denote the moduli space of continuation flow trees by $\M_I(p_1,p_2;p_0|s^t)$ and describe its construction in the following paragraphs.

Each branch in a continuation tree will follow one of the vector fields $V_{i,j;3} = \nabla_{G_{i,j;3}}W_{i,j;3}$ defined in the previous subsection in \ref{nabGWij}. Recall that the path of metrics $g_{i,j;3}^t$ used to define $G_{i,j;3}$ was chosen to be admissible so that the unstable and stable manifolds from each $V_{i,j;3}$ intersect transversely. This does not guarantee the transverse intersection in the flow trees. To fix this and prove that the product does not depend on the perturbation used to achieve transversality, we add the data of a path of perturbation vectors into the construction of the continuation trees.

The perturbation balls $S^0$ for $F^0$ and $S^1$ for $F^1$ might be of different sizes; see Remark \ref{pertdelta}. There is, however, a smooth path $S^t$ of perturbation balls connecting them. To  construct a path $s^t = (s_1^t, s_2^t, s_3^t) \in (S^t)^3$, first pick endpoints $s^0 \in (S^0)^3$ so that $\M_{F^0}(p_1,p_2;p_0'|s^0) \neq \emptyset$ is a smooth manifold for any choice of $p_0' \in \crit_+(w)$ and $s^1 \in (S^1)^3$ so that $\M_{F^1}(p_1',p_2';p_0|s^1) \neq \emptyset$ is a smooth manifold for any choices of $p_1', p_2' \in \crit_+(w)$. By Theorem \ref{cornertrees}, almost every choice of $s^0$ and $s^1$ will suffice for a fixed triplet of critical points so almost every choice still suffices because $\crit_+(w)$ is a finite set. Using these endpoints, construct a smooth path $s^t = (s_1^t, s_2^t, s_3^t) \in (S^t)^3$. This path may be perturbed while keeping the admissible endpoint fixed to achieve transversality as described in the following after we set up perturbed continuation evaluation maps.

To get a manifold structure and compactification results on $\mathcal{M}_I(p_1,p_2; p_0|s^t)$ we need to transversely cut it out of Morse trajectory spaces. As before in Definition \ref{halfinftraj}, we have half-infinite Morse trajectory spaces, for $(p,t) \in C(F^t) \times \{t\}$ for $t=0$ or $t=1$.

$$\M_{i,j;3}((p, t), P_3 \times I) = \{ \gamma: (-\infty, 0] \rightarrow P_3 \times I \mid \dot{\gamma} = V_{i,j;3}, \lim_{s \to -\infty} \gamma(s) = (p,t)\}  \text{ and} $$
$$\M_{i,j;3}(P_3 \times I, (p,t))= \{ \gamma: [0, \infty) \rightarrow P_3 \times I \mid \dot{\gamma} = V_{i,j;3}, \lim_{s \to \infty} \gamma(s) = (p,t) \},$$ 
where we abuse notation and use $p$ to denote a critical point in $\crit_+(w^t)$ and its bijective image in $\crit_+(\w{i}{j}{3}^t)$. Given $p_1, p_2 \in \crit_+(w^0)$ and $p_0 \in \crit_+(w^1),$ let
$$X_I = \M_{1,2;3}((p_1,0), P_3 \times I) \times \M_{2,3;3}((p_2,0), P_3 \times I) \times \M_{1,3;3}((P_3 \times I, (p_0,1)). $$ This space has a smooth structure induced by a diffeomorphism with $$W^-_{(p_1,0)}(V_{1,2;3}) \times W^-_{(p_2,0)}(V_{2,3;3}) \times W^+_{(p_0,1)}(V_{1,3;3}).$$

Given a path $\{s^t\}$ through $(S^t)^3$ as described above, we have analogous perturbed evaluation maps that we will use to construct a map $E_I$ as in Definition \ref{Eijbroken}. Given a half-infinite trajectory $\gamma$ in one of the above spaces, the evaluation maps $\ev{i}{j}{3}{\pm}(\gamma)$  give a point $\gamma(0)= (\gamma(0)|_{P_3},\gamma(0)|_I) \in P_3 \times I$ and we will perturb $\gamma(0)|_{P_3}$ by $s^t$ in $P_3 \times \{t\}$ for $t = {\gamma(0)|_I}$. That is, we have the following three maps, with $\pi$ representing the submersion or the identity to remain consistent with Definition \ref{Eijbroken}.
\begin{align*}
E_{1,2;3}: \M_{1,2;3}((p_1, 0), P_3 \times I) &\to P_3 \times I \\
E_{2,3;3}: \M_{2,3;3}((p_2,0), P_3 \times I) &\to P_3 \times I \\
E_{1,3;3}: \M_{1,3;3}((P_3 \times I, (p_0,1)) &\to P_3 \times I \\
\gamma &\mapsto \pi \left( \gamma(0)|_{P_3} + s^{\gamma(0)|_I} \right)
\end{align*}
Thus, we have the map $E:X_I \to (P_3 \times I)^3$ defined by 
$$E(\gamma_1, \gamma_2, \gamma_3) = (E_{1,2;3}(\gamma_1), E_{2,3;3}(\gamma_2), E_{1,3;3}(\gamma_3)) . $$

\begin{defn}
The \textit{moduli space of continuation trees} is $$\M_I(p_1,p_2;p_0|s^t) \coloneqq E^{-1}(\Delta(P_3 \times I)^3).$$

We may express this moduli space as the following set:
$$\mathcal{M}_I(p_1,p_2; p_0|s^t) = 
  \left\lbrace 	(\gamma_1,\gamma_2,\gamma_3)  \mathrel{}\middle|  \mathrel{} \begin{aligned} 
  & \gamma_1: \left( -\infty, 0 \right] \rightarrow P_3 \times I, \\
  &  \gamma_2: \left( -\infty, 0 \right] \rightarrow P_3 \times I, \\ 
  & \gamma_3: \left[0, \infty \right) \rightarrow  P_3 \times I, \\ 
  & \dfrac{d \gamma_1}{ds} = V_{1,2;3}, \dfrac{d \gamma_2}{ds} = V_{2,3;3}, \dfrac{d \gamma_3}{ds} = V_{1,3;3}, \\
  & E_{1,2;3}(\gamma_1) = E_{2,3;3}(\gamma_2) = E_{1,3;3}(\gamma_3)\\
  & \lim_{s \to -\infty} \gamma_1(s) = (p_1,0), \lim_{s \to -\infty} \gamma_2(s) = (p_2,0),\\
  & \lim_{s \to \infty} \gamma_3(s) = (p_0,1) \end{aligned} \right\rbrace .$$
\end{defn}


\begin{lem}
There is a perturbation of the path $\{s^t\}$ so that $E \pitchfork \Delta(P_3 \times I)^3$. Then $\M_I(p_1,p_2;p_0|s^t)$ is a manifold of dimension $|p_0| - |p_1| - |p_2| + 1$.
\end{lem}

\begin{proof}
The freedom given by perturbing the path $\{s^t\}$ together with the larger class of metrics used to define $V_{i,j;3}$ give us room to achieve transversality. We calculate the dimension:
\begin{align*}
&\text{dim}(E^{-1}(\Delta(P_3 \times I)^3)) \\
 &= \text{dim}(X_I) - \text{codim}((\Delta(P_3 \times I)^3) \\ &= \text{dim}(W^-_{(p_1,0)}(V_{1,2;3})) + \text{dim}(W^-_{(p_2,0)}(V_{2,3;3})) + \text{dim}(W^+_{(p_0,1)}(V_{1,3;3})) - 2(n+3N+1) \\
&=(n+ 3N +1) - \text{ind}_{V_{1,2;3}}((p_1,0)) + (n+ 3N +1) - \text{ind}_{V_{2,3;3}}((p_2,0)) \\ & \quad \quad \quad + \text{ind}_{V_{1,3;3}}((p_0,1)) - 2(n+3N+1)\\
&= \text{ind}_{\w{1}{3}{3}^1}(p_0) + 1 - \text{ind}_{\w{1}{2}{3}^0}(p_1) - \text{ind}_{\w{2}{3}{3}^0}(p_2) \\
&= (|p_0| + 2N) +1 - (|p_1| + N) - (|p_2| +N) \\
&= |p_0| - |p_1| - |p_2| + 1. \qedhere
\end{align*}
\end{proof}


As in previous arguments in this paper, we will need to understand the boundary of the compactification of a 1-dimensional $\mathcal{M}_I(p_1,p_2; p_0|s^t)$. Rather than defining a larger manifold with corners structure as in Section \ref{sec:moduli}, we will use similar arguments to classify possible limits of unbroken continuation trees.

To apply similar arguments, we need bounds on the continuation trees as in Lemmas \ref{lem:cpct-set} and \ref{awayfrom0}. The compact non-linear support set from each generating family $F^t$ gives a path of compact non-linear support sets $K^t$ as in Definition \ref{defK}. A similar argument to Lemma \ref{lem:cpct-set} shows that, for all $\Gamma \subset \mathcal{M}_I(p_1,p_2; p_0|\{s^t\})$, $\text{Im}(\Gamma) \subset \bigcup_{t \in I} (K^t \times \{t\}) \subset P_3 \times I$. Similarly, given $\{\rho^t\}$, the path of smallest positive critical values of $w^t$, we may bound the ```midpoint" of any tree $\Gamma$, which occurs at a specific slice $P_3 \times \{t\}$ away from the critical submanifold of $w_{1,3;3}^t$ as in Lemma \ref{awayfrom0}.

\begin{prop}\label{compconttrees}
Given $p_1, p_2 \in \crit_+(w^0)$ and $p_0 \in \crit_+(w^1)$ with $|p_0| = |p_1| + |p_2|$, if $\{s_t\}$ is a path so that $\mathcal{M}_I(p_1,p_2; p_0|s^t)$ is a 1-manifold, then it may be compactified to a 1-manifold $\overline{\mathcal{M}}_I(p_1,p_2; p_0|s^t)$ with boundary

\begin{align*}
& \partial \overline{\mathcal{M}}_I(p_1,p_2; p_0|s^t) = \\ &\bigcup_{p_1'} \Mtofro{1}{2}{(p_1,0)}{(p_1',0)} \times \mathcal{M}_I(p_1',p_2; p_0|s^t) \; \cup \\
& \bigcup_{p_2'} \Mtofro{2}{3}{(p_2,0)}{(p_2',0)} \times \mathcal{M}_I(p_1,p_2'; p_0|s^t) \; \cup \\
& \bigcup_{p_0'} \mathcal{M}_I(p_1,p_2,p_0'|s^t) \times \Mtofro{1}{3}{(p_0',1)}{(p_0,1)} \; \cup \\
& \bigcup_{p_0''} \mathcal{M}_{F^0}(p_1,p_2,p_0'' | s^0) \times \Mtofro{1}{3}{(p_0'',0)}{(p_0,1)} \; \cup\\
& \bigcup_{p_1'', p_2''} \Mtofro{1}{2}{(p_1,0)}{(p_1'',1)} \times \Mtofro{2}{3}{(p_2,0)}{(p_2'',1)} \times \mathcal{M}_{F^1}(p_1'',p_2''; p_0|s^1),
\end{align*}
where the unions are taken over $p_1' \in C^{|p_1|+1}(F^0)$, $p_2' \in C^{|p_2|+1}(F^0)$, $p_0' \in C^{|p_0|-1}(F^1)$, $p_0'' \in C^{|p_0|}(F^0)$, $p_1'' \in C^{|p_1|}(F^1)$, and $p_2'' \in C^{|p_2|}(F^1)$, respectively.
\end{prop}

\begin{proof}
Let $\M_I(p_1,p_2; p_0|s^t)$ be of dimension 1. A similar argument as in Section \ref{sec:moduli} gives a compactification of this space by trees with once broken branches. By construction of our vector fields, all critical points of $V_{i,j;3}$ live in $P_3 \times \{0\}$ and $P_3 \times \{1\}$. With three branches that may break at critical points in either of these manifolds, we seemingly have six cases of broken trees that might show up in the boundary of a compactified 1-dimensional moduli space of continuation trees:
\begin{enumerate}
\item The branch flowing from $(p_1,0)$ along $V_{1,2;3}$ breaks in $P_3 \times \{0\}$: This would mean that $p_1$ flows along $\nabla_{g_{1,2;3}} \w{1}{2}{3}^0$ to another critical point $p_1' \in C(F^0)$ with $|p_1'|=|p_1|+1$. An index calculation shows that a tree from $(p_1',0)$ and $(p_2,0)$ to $(p_0,1)$ would be isolated. 
\item In the same way, $(p_2,0)$ could flow along $\nabla_{g_{2,3;3}} \w{2}{3}{3}^0$ to a point $(p_2',0)$ with $p_2' \in C^{|p_2|+1}(F^0)$. Note that the indices force only one edge to break in this way at a time.
\item If the branch flowing along $V_{1,3;3}$ ending at $(p_0,1)$ breaks at a point in $P_3 \times \{1\}$, then the trajectories form a tree from $(p_1,0)$ and $(p_2,0)$ to a critical point $(p_0',1)$, where $p_0' \in C^{|p_0|-1}(F^1)$ and then $p_0'$ flows along $\nabla_{g_{1,3;3}} \w{1}{3}{3}^1$ to $p_0$. 
\item If the branch flowing along $V_{1,3;3}$ to $(p_0,1)$ breaks at $t=0$ at a point $(p_0'',0)$, then we see a tree that must be contained in $P_3 \times \{0\}$. Since $\text{ind}_{V_{1,3;3}}(p_0'',0) = \text{ind}_{V_{1,3;3}}(p_0,1) - 1$, it must be that $p_0'' \in C^{|p_0|}(F^0)$. The tree in $P_3 \times \{0\}$ is in a moduli space $\M_{F^0}(p_1, p_2,p_0''|s^0)$ of flow trees from $F^0$, and our conditions on the endpoints of the path $\{s^t\}$ guarantee that this a manifold of dimension $|p_0''| - |p_1| - |p_2| = 0$. We then see a flow line from $(p_0'',0)$ to $(p_0,1)$, which is in the 0-dimensional continuation moduli space $\M_{1,3;3}((p_0'',0), (p_0,1))$.
%
\item Suppose the branch following $V_{1,2;3}$ from $(p_1,0)$ breaks in $P_3 \times \{1\}$ at a point $(p_1'',1)$. This would imply that $(p_1'',1) \in \Crit_+(V_{1,2;3})$ of index $\ind_{\w{1}{2}{3}^1}(p_1)+1$, so $p_1'' \in C^{|p_1|}(F^1)$. Thus, this is a flow line in the isolated continuation moduli space $\M_{1,2;3}((p_1,0),(p_1'',1)) $. Then we see a tree with the $V_{1,2;3}$ branch contained in $P_3 \times \{1\}$. In particular, the ``midpoint" of the tree (the point in $\Delta(P_3 \times I)^3$) is a point $y \in P_3 \times \{1\}$. This means that the branch $\gamma_2$ of the tree that flows along $V_{2,3;3}$ has finite endpoint $\gamma_2(0)$ in $P_3 \times \{1\}$. Due to the $\partial t$ component of the vector field vanishing as $t \to 1$, this cannot happen in finite time. Thus, the branch flowing along $V_{2,3;3}$ must break at a critical point $(p_2'',1)$ with $p_2'' \in C^{|p_2|}(F^1)$. Then there is a tree from $(p_1'',1)$ and $(p_2'',1)$ to $(p_0,1)$ completely contained in $P_3 \times \{1\}$. This tree lives in a moduli space $\mathcal{M}_{F^1}(p_1'',p_2''; p_0|s^1)$, which, due to the construction of the perturbation path $\{s^t\}$, is a manifold of dimension 0. \qedhere
\end{enumerate}
\end{proof}

\begin{defn}
We define a map $K: C^i(F^0) \otimes C^j(F^0) \to C^{i+j-1}(F^1)$ as follows:
\newline
Given $p_1 \in \crit_+(w^0)$ and $p_2 \in \crit_+(w^0)$, then 
$$K(p_1 \otimes p_2) = \sum ( \#_{\Z_2} \mathcal{M}_I(p_1,p_2;p_0|\{s^t\}) ) \cdot p_0 $$ 
where the sum is taken over $p_0 \in \crit_+(w^1)$ with $|p_0|=|p_1|+|p_2|-1$. Extend the product bilinearly over the tensor product.
\end{defn}

The following Corollary follows directly from the description of the boundary of a compactified one-dimensional continuation flow tree moduli space in Proposition \ref{compconttrees}.

\begin{cor}
The map $K: C^i(F^0) \otimes C^j(F^0) \to C^{i+j-1}(F^1)$ is a chain homotopy, i.e.,
$$  \delta_{1,3;3} \circ K + K \circ (\delta_{1,2;3} \otimes 1 + 1 \otimes \delta_{2,3;3}) = (\Phi_{1,3;3} \circ m_2^0) + m_2^1 \circ (\Phi_{1,2;3} \otimes \Phi_{2,3;3})$$
\begin{equation} \label{eqn:cochainhtpy}
     \xymatrixrowsep{5pc} \xymatrixcolsep{8pc}\xymatrix{
       C(F^0) \otimes C(F^0) \ar@<-.5ex>[d]_{m_2^1 \circ (\Phi_{1,2;3}\otimes \Phi_{2,3;3})}\ar@<.5ex>[d]^{\Phi_{1,3;3}\circ m_2^0}  \ar[r]^{\delta \otimes 1 + 1 \otimes \delta} &C(F^0) \otimes C(F^0) \ar@<-.5ex>[d]_{m_2^1 \circ (\Phi_{1,2;3}\otimes \Phi_{2,3;3})}\ar@<.5ex>[d]^{\Phi_{1,3;3}\circ m_2^0} \ar[ld]_{K}  \\
      C(F^1)  \ar[r]_{\delta} &C(F^1).
      }
    \end{equation}

\end{cor}

The following results follow from the fact that $K$ is a chain homotopy:

\begin{thm}\label{inviso}
Let $\leg_t \subset J^1M, t \in [0,1]$ be isotopy of Legendrian submanifolds, and suppose $\leg_0$ has a linear-at-infinity generating family. Then for $F^0$ and $F^1$ guaranteed by Proposition \ref{prop:leg-persist}, the following diagram commutes:
\begin{equation} \label{eqn:legiso}
      \xymatrix{
       GH^*(F^0) \otimes GH^*(F^0) \ar[d]_{\cong}  \ar[r]^{ \quad\quad \mu_2^0} &GH^*(F^0) \ar[d]^{\cong} \\
       GH^*(F^1) \otimes GH^*(F^1) \ar[r]^{\quad \quad \mu_2^1} &GH^*(F^1).
      }
    \end{equation}
\end{thm}

Since paths of metrics and perturbations appeared in the construction of continuation flow trees, the resulting chain homotopy also shows the following two results of invariance.

\begin{cor}\label{indepmetric}
The construction of $\mu_2$ does not depend on choice of metrics from $\mathcal{G}_F$ and $\mathcal{G}_Q$ in Definition \ref{EMS} used in the gradient vector fields.
\end{cor}

\begin{cor}\label{indeppert}
The construction of $\mu_2$ does not depend on choice of perturbation $s$ used to achieve transversality in $\modsp$.
\end{cor}

\bibliographystyle{amsplain} 
\bibliography{main}

\end{document}